\documentclass[a4paper]{amsart}

\usepackage{amsfonts, latexsym, amssymb, graphicx,
multicol, mathrsfs, color, amscd, verbatim, paralist,
xspace, url, euscript, stmaryrd,  amsmath, amscd, enumitem,
bbold, multirow, tikz, mathtools, mathabx}
\usepackage[all,pdf,cmtip]{xy}

\usepackage[colorlinks, linkcolor=blue, citecolor=magenta, urlcolor=cyan]{hyperref}

\usepackage{mathabx}

\def\ra{\rightarrow}

\newtheorem{theorem}{THEOREM}[section]
\newtheorem{corollary}[theorem]{Corollary}
\newtheorem{proposition}[theorem]{Proposition}
\newtheorem{lemma}[theorem]{Lemma}
\newtheorem{conjecture}[theorem]{Conjecture}

\theoremstyle{definition}

\theoremstyle{remark}
\newtheorem{remark}[theorem]{Remark}

\newcommand\C{{\mathbb C}}
\newcommand\CC{{\mathbb C}}
\newcommand\RR{{\mathbb R}}
\newcommand\N{{\mathbb N}}

\newcommand\g{\mathfrak{g}}

\newcommand\h{\mathfrak{h}}
\newcommand\hol{\mathfrak{hol}}
\renewcommand\H{\mathbb{H}}

\newcommand\op[1]{\mathop{\rm #1}\nolimits}

\newcommand\pp{\mathfrak{p}}

\newcommand\R{{\mathbb R}}

\newcommand\Z{{\mathbb Z}}
\newcommand\ZZ{{\mathbb Z}}

\def\SL{\mathop{\rm SL}\nolimits}
\def\CO{\mathop{\rm CO}\nolimits}
\def\SO{\mathop{\rm SO}\nolimits}
\def\SU{\mathop{\rm SU}\nolimits}
\def\U{\mathop{\rm U}\nolimits}
\def\rank{\mathop{\rm rank}\nolimits}
\def\vf{\mathop{\rm vf}\nolimits}
\def\Re{\mathop{\rm Re}\nolimits}
\def\Im{\mathop{\rm Im}\nolimits}

\makeatletter
\def\blfootnote{\xdef\@thefnmark{}\@footnotetext}
\makeatother

\begin{document}

\title[On the symmetry algebras of 5-dimensional CR-manifolds]{On the symmetry algebras of
\vspace{0.1cm}\\
5-dimensional CR-manifolds}
\author[Isaev]{Alexander Isaev}
\author[Kruglikov]{Boris Kruglikov}

\address[Isaev]{Mathematical Sciences Institute\\
Australian National University\\
Acton, ACT 2601, Australia}
\email{alexander.isaev@anu.edu.au}

\address[Kruglikov]{Department of Mathematics and Statistics\\
University of Troms\o{}\\
Troms\o{} 90-37, Norway}
\email{boris.kruglikov@uit.no}

\maketitle

\thispagestyle{empty}

\pagestyle{myheadings}

\begin{abstract}
We show that for a real-analytic connected holomorphically nondegenerate 5-dimensional CR-hypersurface $M$ and its symmetry algebra $\mathfrak{s}$ one has either: (i) $\dim\mathfrak{s}=15$ and $M$ is spherical (with Levi form of signature either $(2,0)$ or $(1,1)$ everywhere), or (ii) $\dim\mathfrak{s}\le11$ where $\dim\mathfrak{s}=11$ can only occur if on a dense open subset $M$ is spherical with Levi form of signature $(1,1)$. Furthermore, we construct a series of examples of pairwise nonequivalent CR-hypersurfaces with $\dim\mathfrak{s}=11$.
\end{abstract}

\section{Introduction}\label{intro}
\setcounter{equation}{0}

A classical problem in geometry is the investigation of automorphism groups and, at the infinitesimal level, of symmetry algebras for classes of manifolds endowed with geometric structures of fixed type. Given such a class $\mathcal{C}$ of manifolds, the {\it symmetry algebra of $M\in{\mathcal C}$}\, is the Lie algebra ${\mathfrak s}(M)$ of vector fields on $M$ whose local flows preserve the structure, and its dimension $\dim{\mathfrak s}(M)$ is called the {\it symmetry dimension of $M$}. In particular, an important question is to determine the maximal value $D_{\hbox{\tiny\rm max}}$ of the symmetry dimension over all $M\in\mathcal{C}$ as well as its possible values close to $D_{\hbox{\tiny\rm max}}$. In many situations this problem has been extensively studied and the maximally symmetric $M\in\mathcal{C}$ (i.e., those with $\dim{\mathfrak s}(M)=D_{\hbox{\tiny\rm max}}$) have been explicitly classified.

While describing large symmetry dimensions, one often encounters a {\it gap phenomenon}, that is, the nonrealizability of some of the values immediately below $D_{\hbox{\tiny\rm max}}$ as $\dim{\mathfrak s}(M)$ for any $M\in{\mathcal C}$. One then searches for the next realizable value, the {\it submaximal dimension $D_{\hbox{\tiny\rm smax}}$}, thus obtaining the interval $(D_{\hbox{\tiny\rm smax}}, D_{\hbox{\tiny\rm max}})$ called the {\it first gap}, or {\it lacuna}, for the symmetry dimension. The lacunary behavior of $\dim{\mathfrak s}(M)$ may extend further, and, ideally, one would like to determine all such lacunas as well as to characterize---to the greatest extent possible---the manifolds with sufficiently large nonlacunary values of $\dim{\mathfrak s}(M)$.

The best-known case for which the above program has been implemented with much success, both in the global and infinitesimal settings, is Riemannian geometry where ${\mathcal C}$ is the class of all smooth\footnote{In this paper smoothness is always understood as $C^{\infty}$-smoothness.} connected Riemannian manifolds of dimension $n\ge 2$. In this situation, $\mathfrak{s}(M)$ is the Lie algebra of all smooth vector fields on $M$ whose flows consist of local isometries,  $D_{\hbox{\tiny\rm max}}=n(n+1)/2$, and the manifolds $M$ satisfying $\dim{\mathfrak s}(M)=D_{\hbox{\tiny\rm max}}$ are the spaces of constant curvature. Furthermore, a number of lacunas for the symmetry dimension have been identified, and the manifolds with sufficiently high nonlacunary values of $\dim{\mathfrak s}(M)$ have been shown to admit reasonable descriptions (see, e.g., \cite{F}, \cite{Eg1}, \cite[p.~ 219]{Eg2}, \cite[Section 3]{I2}, \cite[Chapter 2]{Ko}, \cite{KN}, \cite{W}). For other geometric structures, results of this kind can be found, e.g., in \cite{I2}, \cite{Ko}, \cite{Kr}, \cite{KT}, \cite{Ma}, \cite{Tr}.

In this paper we turn to CR-geometry, in which case much less is known about the behavior of the symmetry dimension. Recall that an {\it almost CR-structure}\, on a smooth manifold $M$ is a subbundle $H(M)\subset T(M)$ of the tangent bundle of even rank, called the CR-subbundle, or CR-distribution, endowed with operators of complex structure $J_p:H_p(M)\ra H_p(M)$, $J_p^2= -\hbox{id}$, that smoothly depend on $p\in M$. A manifold equipped with an almost CR-structure is called an {\it almost CR-manifold}. The number $\rank(H(M))/2$ is denoted by $\hbox{CRdim}\, M$ and called the CR-dimension of  $M$. The complementary dimension $\dim M-2\hbox{CRdim}\, M$ is called the CR-codimension of $M$. 

Next, for every $p\in M$ we have
$
H_p(M)\otimes\C=H_p^{(1,0)}(M)\oplus H_p^{(0,1)}(M),
$
where
$$
\begin{array}{l}
H_p^{(1,0)}(M):=\{X-iJ_pX\mid X\in H_p(M)\},\\
\vspace{-0.3cm}\\
H_p^{(0,1)}(M):=\{X+iJ_pX\mid X\in H_p(M)\}.
\end{array}
$$
Then the almost CR-structure on $M$ is said to be integrable if the bundle $H^{(1,0)}(M)$ is involutive, i.e., for any pair of local
sections ${\mathfrak z},{\mathfrak z}'$ of $H^{(1,0)}(M)$ the commutator $[{\mathfrak z},{\mathfrak z}']$
is also a local section of $H^{(1,0)}(M)$. An integrable almost CR-structure is called a {\it CR-structure}\, and a manifold equipped with a CR-structure a {\it CR-manifold}. In this paper we consider only {\it CR-hypersurfaces}, i.e., CR-manifolds of CR-codimension 1.

If $M$ is a real hypersurface in a complex manifold ${\mathcal M}$ with operators of almost complex structure ${\mathcal J}_q$, $q\in {\mathcal M}$,  it is naturally an almost CR-manifold with $H_p(M):=T_p(M)\cap {\mathcal J}_p(T_p(M))$ and $J_p$ being the restriction of ${\mathcal J}_p$ to $H_p(M)$ for every $p\in M$. Moreover, the almost complex structure so defined is integrable, thus $M$ is in fact a CR-hypersurface of CR-dimension $\dim_{\C}{\mathcal M}-1$.

Further, the {\it Levi form}\, of a CR-hypersurface $M$ comes from taking commutators of local sections of
$H^{(1,0)}(M)$ and $H^{(0,1)}(M)$. Let $p\in M$, $\zeta,\zeta'\in
H_p^{(1,0)}(M)$. Choose  local sections ${\mathfrak z}$, ${\mathfrak z}'$ of $H^{(1,0)}(M)$ near
$p$ such that ${\mathfrak z}(p)=\zeta$, ${\mathfrak z}'(p)=\zeta'$. The Levi form of $M$ at
$p$ is then the Hermitian form on $H_p^{(1,0)}(M)$ with values in the space $(T_p(M)/H_p(M))\otimes\CC$ given by
${\mathcal L}_M(p)(\zeta,\zeta'):=i[{\mathfrak z},\overline{{\mathfrak z}'}](p)\,(\hbox{mod}\, H_p(M)\otimes\C)$. For fixed $\zeta$ and $\zeta'$ the right-hand side of this formula is independent of the choice of ${\mathfrak z}$ and ${\mathfrak z}'$, and, identifying $T_p(M)/H_p(M)$ with $\R$, one obtains a $\C$-valued Hermitian form defined up to a real scalar multiple.

As shown in classical work \cite{C}, \cite{CM}, \cite{Ta1}--\cite{Ta3}, \cite{Tr} (see also \cite{BS}), the dimension of the symmetry algebra $\mathfrak{s}(M)$ of a {\it Levi-nondegenerate}\, connected CR-hypersurface $M$ of CR-dimension $n$ does not exceed $n^2+4n+3$. Furthermore, $\dim\mathfrak{s}(M)=n^2+4n+3$ implies that $M$ is {\it spherical}, i.e., that near its every point $M$ is CR-equivalent to an open subset of the quadric
\begin{equation}
Q_{K}:=\Bigl\{(z_1,\dots,z_{n+1})\in\CC^{n+1}:  \Re(z_{n+1})=\sum_{j=1}^{K}(\Re(z_j))^2-\sum_{j=K+1}^{n}(\Re(z_j))^2\Bigr\}\label{QK}
\end{equation}
for some $n/2\le K\le n$. The Levi form of $Q_{K}$ has signature $(K, n-K)$ everywhere and $\dim\mathfrak{s}(Q_K)=n^2+4n+3$ for all $K$. Thus, for the class of Levi-nondegenerate connected CR-hypersurfaces of CR-dimension $n$ one has $D_{\hbox{\tiny\rm max}}=n^2+4n+3$. Despite the above result being classical, the submaximal value $D_{\hbox{\tiny\rm smax}}$ in the Levi-nondegenerate setting has only been recently computed. Namely, in \cite{Kr} it was shown that $D_{\hbox{\tiny\rm smax}}=n^2+3$ in the strongly pseudoconvex (Levi-definite) case and $D_{\hbox{\tiny\rm smax}}=n^2+4$ in the Levi-indefinite case. It is instructive to compare this result with known bounds on the dimension of the stability group (see \cite{EI} and references therein for details).

It should be noted that the geometry of Levi-nondegenerate CR-hypersurfaces is a particular instance of parabolic geometry and that the classical results stated above can be derived within the parabolic geometry framework (see, e.g., \cite{CSc}). Likewise, the argument of \cite{Kr} is based on article \cite{KT} where gap phenomena for general parabolic geometries were studied. For an extensive introduction to parabolic geometry we refer the reader to monograph \cite{CSl}.

In the absence of Levi-nondegeneracy, finding the maximal and submaximal dimensions of the symmetry algebra is much harder. To simplify the setup, in this case one usually switches to the real-analytic category by assuming the manifolds and the vector fields forming the symmetry algebra to be real-analytic rather than just smooth. In order to guarantee the finite-dimensionality of $\mathfrak{s}(M)$ it then suffices to require that $M$ be {\it holomorphically nondegenerate}\, (see \cite[\S 11.3, \S12.5]{BER}, \cite{Eb}, \cite{S}). Regarding the maximal possible value for $\dim\mathfrak{s}(M)$ in this situation, we mention the following variant of a conjecture due to V.~Beloshapka (cf.~ \cite[p.~38]{B}):    

\begin{conjecture}\label{beloshapka} For any real-analytic connected holomorphically nondegenerate CR-hypersurface $M$ of CR-dimension $n$ one has $\dim\mathfrak{s}(M)\le n^2+4n+3$, with the maximal value $n^2+4n+3$ attained only if on a dense open set $M$ is spherical.
\end{conjecture}

For $n=1$ the above conjecture holds true since a 3-dimensional holomorphically nondegenerate CR-hypersurface always has points of Levi-nondegeneracy. For $n=2$ the conjecture was established in \cite[Corollary 5.8]{IZ} where the proof relied on the reduction of 5-dimensional uniformly Levi-degenerate 2-nondegenerate CR-structures to absolute parallelisms. Thus, for real-analytic connected holomorphically nondegenerate CR-hypersurfaces of CR-dimen\-sion $n=1$, $2$ one has,  just as in the Levi-nondegenerate case, $D_{\hbox{\tiny\rm max}}=n^2+4n+3$. Note, however, that this maximal value is not available for $n\ge 3$ as Conjecture \ref{beloshapka} remains open in this case.

Next, it was shown in \cite{KS} that for $n=1$ the condition $\dim\hol(M,p)>5$ for $p\in M$ implies that $M$ is spherical near $p$, where $\hol(M,p)$ is the Lie algebra of germs at $p$ of real-analytic vector fields on $M$ whose flows consist of CR-transformations. The authors called this statement the Dimension Conjecture and argued that it can be viewed as a variant of H.~Poincar\'e's {\it probl\`eme local}. It is also related to characterizations of 3-dimensional CR-hypersurfaces in terms of the dimension of their stability groups (see \cite{KL} and references therein). The method of proof proposed in \cite{KS} is rather involved and based on considering second-order complex ODEs with meromorphic singularity. On the other hand, in our recent paper \cite{IK} we gave a single-page proof of the conjecture by using only known facts on Lie algebras and their actions. Furthermore, applying the argument of \cite{IK} to the symmetry algebra $\mathfrak{s}(M)$ instead of $\hol(M,p)$, one obtains $D_{\hbox{\tiny\rm smax}}=5$.

In the present paper we focus on the case $n=2$, i.e., on the case of real-analytic connected holomorphically nondegenerate CR-hypersurfaces of dimension 5. One of our goals is to determine the submaximal dimension $D_{\hbox{\tiny\rm smax}}$ in this situation. In our first main result, Theorem \ref{main1}, we show that one has either $\dim\mathfrak{s}(M)=n^2+4n+3=15$ and $M$ is spherical (with Levi form of signature $(2,0)$ or $(1,1)$ everywhere), or $\dim\mathfrak{s}(M)\le11$ with the equality $\dim\mathfrak{s}(M)=11$ occurring only if on a dense open subset $M$ is spherical with Levi form of signature $(1,1)$. In particular, this result improves on the statement of Conjecture \ref{beloshapka} for $n=2$ and yields $D_{\hbox{\tiny\rm smax}}\le 11$. As in the short proof of the main theorem of \cite{KS} given in \cite{IK}, our argument relies on Lie algebra techniques, notably on the description of proper subalgebras of maximal dimension in $\mathfrak{su}(1,3)$ and $\mathfrak{su}(2,2)$ obtained in Propositions \ref{su13}, \ref{su22}. 

Further, in Theorem \ref{main2} we give a series of examples of pairwise nonequivalent CR-hypersurfaces with $\dim\mathfrak{s}(M)=11$ thus proving that $D_{\hbox{\tiny\rm smax}}=11$ (see Corollary \ref{smaxdim}). The examples of Theorem \ref{main2} also lead to the following analogue of the result of \cite{KS} for $n=2$: the condition $\dim\hol(M,p)>11$ for $p\in M$ implies that $M$ is spherical near $p$, and this estimate is sharp (see Corollary \ref{dimconjnew}). The examples of Theorem \ref{main2} are quite nontrivial, and in Section \ref{fourdimorbit} we explain in detail how our search for them was organized. Our strategy is summarized in Theorem \ref{examplegeneral}, which is a result of independent interest. In fact, our methods point towards a potential classification of all CR-hypersurfaces with $\dim\mathfrak{s}(M)=11$ (see Remark \ref{higherdegreeexamples} for details). 

Finally, in Theorem \ref{realizableimensions} we show that every integer between 0 and 10 is also realizable as the symmetry dimension of a real-analytic connected holomorphically nondegenerate 5-dimensional CR-hypersurface. 

\noindent{\bf Acknowledgements.} This work was initiated while the first author was visiting the University of Troms\o{}. Significant progress was made while the first author was visiting Boston College and the second author was visiting the Australian National University. We thank the above institutions for their hospitality. We also thank M. Kol\'a\v r for useful discussions, I. Kossovskiy for comments on an earlier version of the paper, and the referee for helpful remarks. The research is supported by the Australian Research Council. We acknowledge the use of {\tt Maple}'s {\tt DifferentialGeomet\-ry} package for most of our calculations.

\section{CR-hypersurfaces with large symmetry algebras}\label{first}
\setcounter{equation}{0}

\subsection{Subalgebras of maximal dimension in $\mathfrak{su}(1,3)$ and $\mathfrak{su}(2,2)$} 

We start by proving two algebraic propositions.

\begin{proposition}\label{su13} For any proper subalgebra $\h$ of $\g:=\mathfrak{su}(1,3)$ one has $\dim\h\le10$, and in the
case of equality $\h$ is conjugate to the parabolic subalgebra $\pp_{13}$, which is equal to $\g_{\ge0}$ in the
contact grading $\g=\g_{-2}\oplus\g_{-1}\oplus\g_0\oplus\g_1\oplus\g_2$.
\end{proposition}

\begin{remark}\label{remp13} Abstractly, we have $\pp_{13}=\g_0\oplus\g_1\oplus\g_2=
(\mathfrak{su}(2)\oplus\R^2)\oright\mathfrak{heis}_5$, where $\mathfrak{heis}_5:=\g_1\oplus\g_2=\C^2\oplus\R$
(with the natural $\g_0$-module structure) is the Heisenberg algebra. This is the only parabolic subalgebra of $\g$ up to conjugation.
\end{remark}

\begin{proof} By Mostow's theorem (see \cite{Mo}), a maximal subalgebra of a semisimple Lie algebra is either parabolic, or
pseudotoric (the stabilizer of a pseudotorus), or semisimple. The first case yields, up to conjugation, the parabolic subalgebra $\pp_{13}$. In the second case, a pseudotorus is a 1- or 2-dimensional subalgebra of a Cartan subalgebra of $\g$. The stabilizer of such a subalgebra has maximal dimension in the 1-dimensional case when it is conjugate to either $\mathfrak{u}(3)$ or $\mathfrak{u}(1,2)$. Hence, the dimension of a pseudotoric subalgebra does not exceed 9.

The last case, when $\h$ is semisimple, is the most complicated. First, we list all simple real Lie algebras of dimension less than $\dim\g=15$:\footnote{Here for a complex Lie algebra ${\mathfrak a}$, we denote by ${\mathfrak a}_{{}_{\R}}$ the underlying real Lie algebra. We also use this convention for complex Lie groups.}
\begin{equation}
\begin{array}{rl}
(3D) &\quad  \mathfrak{su}(2)=\mathfrak{so}(3,\R)=\mathfrak{sp}(1),\
\mathfrak{su}(1,1)=\mathfrak{sl}(2,\R)=\mathfrak{sp}(2,\R);\\
\vspace{-0.3cm}\\
(6D) &\quad  \mathfrak{sl}(2,\C)_\R=\mathfrak{so}(1,3);\\
\vspace{-0.3cm}\\
(8D) &\quad  \mathfrak{su}(3),\ \mathfrak{su}(1,2),\ \mathfrak{sl}(3,\R);\\
\vspace{-0.3cm}\\
(10D) &\quad  \mathfrak{sp}(2)=\mathfrak{so}(5),\ \mathfrak{sp}(1,1)=\mathfrak{so}(1,4),\
\mathfrak{sp}(4,\R)=\mathfrak{so}(2,3);\\
\vspace{-0.3cm}\\
(14D) &\quad  \g_2^c=\op{Lie}(G_2^c),\ \g_2^*=\op{Lie}(G_2^*).
\end{array}\label{list}
\end{equation}

If $\dim\h=14$, then $\h$ could be one of the simple Lie algebras shown in $(14D)$ but this cannot happen as then
$\h$ would have a nontrivial representation $\g/\h$ of dimension 1. Indeed, if the representation were trivial, then due to the existence of an invariant complement, the algebra $\g=\h\oplus\R$ would not be simple. Alternatively, $\h$ could be the direct sum of several simple algebras from the above list due to
the partitions $14=6+8=3+3+8$ but any such sum is a semisimple Lie algebra of rank 4 and thus cannot be 
embedded in the semisimple algebra $\g$ of rank 3.

If $\dim\h=13$, then $\h$ could be the direct sum of a 3- and 10-dimensional algebras from the list above, but $\g$ contains no $\mathfrak{sp}$-subalgebras of rank 2 shown in $(10D)$. Indeed, any such subalgebra would have to have a representation on $\C^4$ 
endowed with an invariant Hermitian form of signature $(1,3)$,\footnote{Throughout the paper we say that a Hermitian form $H$ on $\C^n$ has signature $(p,q)$, with $p+q=n$, if the bilinear form $\Re(H)$ has signature $(2p,2q)$.} which cannot occur in any of the three 10-dimensional possibilities: in the first 
two cases the signature is $(4,0)$ and $(2,2)$, respectively, for the standard representations $\C^{4,0}$ and $\C^{2,2}$; in the last case the standard representation is $\R^4$, and the complex representation of minimal dimension is $\C^4=\R^4\otimes\C$ where the signature is $(2,2)$.

If $\dim\h=12$, then $\h$ could be the direct sum of several simple algebras from the above list due to the
partitions $12=6+6=3+3+6=3+3+3+3$, but any such sum is a semisimple Lie algebra of rank 4 and thus cannot be
embedded in the semisimple algebra $\g$ of rank 3.

If $\dim\h=11$, then $\h$ could be the direct sum of a 3- and 8-dimensional algebras from the list above. Both $\mathfrak{su}(3)$ and $\mathfrak{su}(1,2)$ naturally embed in $\g=\mathfrak{su}(1,3)$, and this is the only embedding up to conjugation (the only complex representation of each of $\mathfrak{su}(3)$ and $\mathfrak{su}(1,2)$ on $\C^4$ is, up to equivalence, the direct sum of the standard 3-dimensional complex representation and the trivial 1-dimensional one).
The centralizer of each of $\mathfrak{su}(3)$ and $\mathfrak{su}(1,2)$ in $\g$ is 1-dimensional, which is apparent in the embeddings
$\mathfrak{u}(3),\mathfrak{u}(1,2)\hookrightarrow\mathfrak{su}(1,3)$. Thus, neither $\mathfrak{su}(3)$ nor $\mathfrak{su}(1,2)$ admits a 3-dimensional direct summand in $\g$. Finally, the only nontrivial complex representation of the algebra
$\mathfrak{sl}(3,\R)$ of complex dimension not exceeding 4 is the complexified standard representation $\C^3=\R^3\otimes\C$, but it does 
not possess any invariant Hermitian form, hence there is no embedding of $\mathfrak{sl}(3,\R)$ in $\g$. This completes the proof. \end{proof}

\begin{proposition}\label{su22} For any proper subalgebra $\h$ of $\g:=\mathfrak{su}(2,2)$ one has $\dim\h\le11$, and in the
case of equality $\h$ is conjugate to the parabolic subalgebra $\pp_2$, which is isomorphic to $\g_{\ge0}$ in the
$|1|$-grading $\g=\g_{-1}\oplus\g_0\oplus\g_1$.
\end{proposition}

\begin{remark}\label{remp22} Abstractly, we have $\pp_2=\g_0\oplus\g_1=\mathfrak{co}(1,3)\oright\R^{1,3}$, where 
$\mathfrak{co}(1,3):=\mathfrak{so}(1,3)\oplus\R=\mathfrak{sl}(2,\C)_{\R}\oplus\R$. This subalgebra will be considered in detail in Section \ref{algp2} below. There also exists another parabolic subalgebra of $\g$, of dimension 10, namely the analogue of the parabolic subalgebra from Proposition \ref{su13} written as $\pp_{13}=\g_0\oplus\g_1\oplus\g_2=(\mathfrak{su}(1,1)\oplus\R^2)\oright\mathfrak{heis}_5$ 
in the contact grading, where $\mathfrak{heis}_5:=\C^{1,1}\oplus\R$. These are the only parabolic subalgebras of $\g$ up to conjugation.
\end{remark}

\begin{proof} We follow the same idea as in the proof of Proposition \ref{su13}, but in this situation there are more subtleties in describing the maximal algebras of $\g$. Again, using Mostow's theorem we get either the parabolic subalgebras $\pp_2$, $\pp_{13}$, or pseudotoric subalgebras (whose dimensions are again at most 9), or  semisimple subalgebras. Thus, we now assume that $\h$ is semisimple and use list (\ref{list}).

The case $\dim\h=14$ is completely analogous to that in the proof of Proposition \ref{su13}. If $\dim\h=13$, then $\h$ could be the direct sum of a 3- and 10-dimensional algebras, and this time $\g$ does have $\mathfrak{sp}$-subalgebras of rank 2. Notice first that $\mathfrak{sp}(2)$ cannot be embedded in $\g$ since any maximal compact subalgebra of $\g$ is conjugate to $\mathfrak{so}(4)=\mathfrak{su}(2)\oplus\mathfrak{su}(2)$, a subalgebra of dimension $6<\dim\mathfrak{sp}(2)=10$. On the other hand, both 
$\mathfrak{sp}(1,1)$ and $\mathfrak{sp}(4,\R)$ do embed in $\mathfrak{su}(2,2)$ as can be seen via their standard representations. Indeed, $\mathfrak{sp}(1,1)$ acts on $\H^{1,1}$, which can be identified with $\C^{2,2}$ by forgetting the quaternionic structure except for $i$, and there is an invariant Hermitian form. Also, the standard action of the algebra $\mathfrak{sp}(4,\R)$ on $\R^4$ yields an action on $\C^4$ and leads to a unique (up to equivalence) nontrivial complex representation of minimal dimension. The symplectic structure $\omega_0$ on $\R^4$ determines the invariant Hermitian form $g+i\omega$ on $\C^4=\R^4+i\R^4$ of signature $(2,2)$ by the formulas $g(x_1+ix_2,y_1+iy_2):=\omega_0(x_1,y_2)-\omega_0(x_2,y_1)$, $\omega:=\omega_0\oplus\omega_0$,
and this gives the embedding $\mathfrak{sp}(4,\R)\hookrightarrow\mathfrak{su}(2,2)$.
However, a direct computation shows that the centralizer of each of $\mathfrak{sp}(1,1)$, $\mathfrak{sp}(4,\R)$ 
in $\g$ is trivial, hence neither subalgebra admits a 3-dimensional direct summand in $\g$.

The case $\dim\h=12$ is analogous to that in the proof of Proposition \ref{su13}, and for $\dim\h=11$ the only 
semisimple subalgebra could be the direct sum of a 3-dimensional and 8-dimensional algebras. Among the latter only $\mathfrak{su}(1,2)$
embeds in $\g$, and the centralizer of this embedding is easily seen to be 1-dimensional. Thus, $\mathfrak{su}(1,2)$ does not admit a 3-dimensional direct summand in $\mathfrak{g}$. The proof is complete. 

\end{proof}

\begin{remark}\label{su22dim10} In parallel with Proposition \ref{su13} we note that every 10-dimensional subalgebra of $\mathfrak{su}(2,2)$ is conjugate to one of $\pp_{13}$ (as described in Remark \ref{remp22}), $\mathfrak{{so}}(1,3)\oright\R^{1,3}\subset\pp_{2}$, $\mathfrak{sp}(1,1)$, $\mathfrak{sp}(4,\R)$.  
\end{remark}

\subsection{A gap phenomenon for the symmetry algebra of a 5-dimensional CR-hypersurface}

Let first $M$ be a connected smooth CR-hypersurface. Recall that an infinitesimal CR-automorphism of $M$ is a smooth vector field on $M$ whose flow consists of CR-transformations. We denote the Lie algebra of all such vector fields by $\mathfrak{aut}(M)$ and the Lie algebra of germs of infinitesimal CR-automorphisms of $M$ at a point $p$ by $\mathfrak{aut}(M,p)$. 

From now on, we assume that $M$ is real-analytic. The main object of our study is the Lie subalgebra $\hol(M)\subset\mathfrak{aut}(M)$ of {\it real-analytic}\, infinitesimal CR-automorphisms of $M$. In the real-analytic category, $\hol(M)$ is exactly the symmetry algebra $\mathfrak{s}(M)$ of $M$ and often, when there is no fear of confusion, we denote it simply by $\mathfrak{s}$. For $p\in M$ one can also consider the Lie subalgebra $\hol(M,p)\subset\mathfrak{aut}(M,p)$ of germs of real-analytic infinitesimal CR-automorphisms of $M$ at $p$.  Clearly, $\hol(M)$ may be viewed as a subalgebra of $\hol(M,p)$ for any $p$. By \cite[Theorem 1.12]{AF}, the CR-hypersurface $M$ admits a closed real-analytic CR-embedding as a hypersurface in a complex manifold ${\mathcal M}$, and it is not hard to show (see, e.g., \cite[Proposition 12.4.22]{BER}) that every real-analytic infinitesimal CR-automorphism defined on an open subset $U\subset M$ is the real part of a holomorphic vector field defined on an open subset ${\mathcal U}\subset {\mathcal M}$ with $U\subset{\mathcal M}\cap{\mathcal U}$. In what follows we will often speak about an element of $\hol(M)$ either as a holomorphic vector field, say $V$, defined near $M$ in ${\mathcal M}$ or as $2\Re(V)$ restricted to $M$, without mentioning this difference explicitly. 

We always assume $M$ to be {\it holomorphically nondegenerate}. The condition of holomorphic nondegeneracy for a real-analytic hypersurface in complex space was introduced in \cite{S} and requires that for every point of the hypersurface there exists no nontrivial holomorphic vector field tangent to the hypersurface near the point. Discussions of this condition can be found in \cite[\S 11.3]{BER}, \cite{Eb} but for the purposes of this paper we will only require the fact, stated in \cite[Corollary 12.5.5]{BER}, that the holomorphic nondegeneracy of $M$ is equivalent to the finite-dimensionality of all the algebras $\hol(M,p)$. Notice that together with \cite[Proposition 12.5.1]{BER} this corollary implies that the finite-dimensionality of $\hol(M,p_0)$ for some $p_0\in M$ implies  the finite-dimensionality of $\hol(M,p)$ for all $p\in M$. It is clear that for a holomorphically nondegenerate $M$ the symmetry algebra $\hol(M)$ is finite-dimensional. Also, in this case for every $p\in M$ there exists a connected neighborhood $U$ of $p$ in $M$ for which the natural map $\hol(U)\to\hol(M,p)$ is surjective; for any such $U$ we have $\hol(M,p)=\hol(U,p)=\hol(U)$.

In the present paper we focus on the case $\dim M=5$. Recall from the introduction (see Conjecture \ref{beloshapka}) that in this situation $D_{\hbox{\tiny\rm max}}:=\max_M\dim \hol(M)=15$ and that $\dim \hol(M)=15$ implies that $M$ is spherical on a dense open subset. In our first main result below, we improve on this statement by demonstrating  that $\dim \hol(M)=15$ in fact yields that $M$ is spherical {\it everywhere} and also show that for the symmetry algebra a {\it gap phenomenon}\, occurs, namely, that several values immediately below the maximal value $15$ are not realizable as $\dim\hol(M)$ for\linebreak any $M$.

\begin{theorem}\label{main1} Assume that a real-analytic connected {\rm 5}-dimensional CR-hyper\-surface $M$ is holomorphically nondegenerate. Then for its symmetry algebra $\mathfrak{s}=\hol(M)$ one has either $\dim\mathfrak{s}=15$ and $M$ is spherical {\rm (}with Levi form of signature $(2,0)$ or $(1,1)$ everywhere{\rm )}, or $\dim\mathfrak{s}\le11$. Furthermore, if $\dim\mathfrak{s}=11$, then on a dense open subset $M$ is spherical with Levi form of signature $(1,1)$.
\end{theorem}

\begin{proof} There exists a proper real-analytic subset $V\subset M$ such that the complement $M\setminus V$ is either (i) Levi nondegenerate or (ii) uniformly Levi-degenerate of rank 1 and 2-nondegenerate (see \cite[\S 11.1]{BER} for the definition of $k$-nondegeneracy). In case (ii), by \cite[Corollary 5.4]{IZ} one has $\dim\mathfrak{s}\le10$. In case (i), if $M$ is nonspherical at some $p\in M\setminus V$, by \cite{Kr} we have $\dim\mathfrak{s}\le 8$. Thus, using Propositions \ref{su13}, \ref{su22}, we conclude that one of the following possibilities occurs: (a) $\dim\mathfrak{s}=15$, the manifold $M\setminus V$ is spherical with Levi form of signature $(2,0)$ (resp.~of signature $(1,1)$) everywhere, and $\mathfrak{s}=\mathfrak{su}(1,3)$ (resp.~$\mathfrak{s}=\mathfrak{su}(2,2)$), or (b) $\dim\mathfrak{s}\le11$, and the equality holds only if $M\setminus V$ is spherical with Levi form of signature $(1,1)$.

To prove the theorem, we only need to consider case (a) and show that $M$ is spherical at the points of $V$. Let $S:=\SU(1,3)$ (resp.~$S:=\SU(2,2)$) provided $\mathfrak{s}=\mathfrak{su}(1,3)$ (resp.~$\mathfrak{s}=\mathfrak{su}(2,2)$). If the orbit of every point of $V$ under the corresponding local action of $S$ is open, then $M$ is spherical as required. Another possibility is the existence of a local $S$-orbit $\Gamma$ in $V$. As $\Gamma$ has positive codimension in $M$ and the orbit of every point in $M\setminus V$ is open due to sphericity, $\Gamma$ is a singular orbit. Let us prove that no such orbit can in fact occur. Locally near $p_0\in\Gamma$ we have $\Gamma=S/R$, thus the Lie algebra $\mathfrak{s}$ has a subalgebra $\mathfrak{r}$ with $11\le\dim\mathfrak{r}\le15$. If $\mathfrak{s}=\mathfrak{su}(1,3)$, by Proposition \ref{su13} the only possibility is $\mathfrak{r}=\mathfrak{s}$. This means that $\Gamma=\{p_0\}$, so the action has a fixed point. Then by Guillemin-Sternberg's theorem (see \cite[pp.~113--115]{GS}), the action of the semisimple algebra $\mathfrak{su}(1,3)$ is linearizable near $p_0$, and we obtain a nontrivial 5-dimensional representation of $\mathfrak{su}(1,3)$. But the lowest-dimensional representation of $\mathfrak{su}(1,3)$ is the standard $\C^{1,3}$ of real dimension 8, which is a contradiction.

Similarly, for $\mathfrak{s}=\mathfrak{su}(2,2)$ in the case $\mathfrak{r}=\mathfrak{s}$ we obtain a contradiction.
However, in this situation, by Proposition \ref{su22}, the subalgebra $\mathfrak{r}$ can be also conjugate to the subalgebra $\mathfrak{p}_2$ of dimension 11. In this case $\Gamma$ is 4-dimensional, and, considering the action of $\g_0=\mathfrak{co}(1,3)$ on $\g_{-1}$ in the $|1|$-grading on $\mathfrak{s}$, we notice that $\mathfrak{so}(1,3)\subset\g_0\subset\mathfrak{r}\subset\mathfrak{s}$ fixes the point $p_0$ and acts on $T_{p_{{}_0}}(\Gamma)$ as on the standard representation $\R^{1,3}$, up to an automorphism (see Section \ref{algp2} for details). Since $\mathfrak{so}(1,3)$ is simple, we again invoke Guillemin-Sternberg's theorem and obtain a nontrivial linearization near $p_0\in M$. This yields a representation of $\mathfrak{so}(1,3)$ that is the sum of the representation on $\R^{1,3}$ and the trivial one. Thus, no local orbit of $\SO(1,3)$ near the point $p_0$ is open (in fact all such orbits have codimension greater than 1). 

Recall now that due to the sphericity of $M\setminus V$, for every point in this subset the isotropy subalgebra is (a conjugate of) the subalgebra $\mathfrak{p}_{13}\subset\mathfrak{su}(2,2)$ described in Remark \ref{remp22}. On the other hand, it is straightforward to see that for a generic element $g\in \SU(2,2)$ arbitrarily close to the identity, the Lie subalgebra $\mathfrak{so}(1,3)\subset{\mathfrak{su}}(2,2)$ is transversal (as a vector space) to the conjugate $\op{Ad}_g\mathfrak{p}_{13}$, i.e., $\dim\mathfrak{so}(1,3)\cap \op{Ad}_g\mathfrak{p}_{13}=1$. Therefore, arbitrarily close to $p_0$ one can find a point whose local $\SO(1,3)$-orbit is open in $M$, which contradicts our earlier conclusion. We have thus shown that the local action of $S$ has no singular orbits, and the proof is complete.

\end{proof}

\begin{remark}\label{anotherproof} One can replace the argument in the last paragraph of the proof of Theorem \ref{main1} with the following argument. As before, we linearize the action of $\mathfrak{so}(1,3)$ near $p_0$ and, in addition, notice that the local action of the group $\CO(1,3)$ on $(\Gamma,p_0)\simeq(\R^4,0)$ has one closed orbit (the null-cone) and three open ones: the positive and negative timelike orbits and the spacelike orbit. Consider a point in $M$ that, with respect to the linearizing coordinates, lies in the positive timelike region. It has a neighborhood $U\subset M$ foliated by local 4-dimensional $\CO(1,3)$-orbits, and we call this foliation $\mathcal{F}$. The isotropy subalgebra of any point $p\in U$ is conjugate to $\mathfrak{so}(3)$, and the isotropy representation on $T_p(\mathcal{F})=\R^4$ is, up to an automorphism, the sum of the standard and trivial representations. As this representation preserves no complex structure, the foliation $\mathcal{F}$ cannot be complex anywhere, i.e., the CR-distribution $H(M)$ intersects $\mathcal{F}$ transversally everywhere. Consider the distribution $L:=H(M)\cap T(\mathcal{F})$ of rank 3 and the complex line distribution $\Pi:=L\cap JL$. Observe that $\Pi$ is $\CO(1,3)$-invariant. On the other hand, it is not hard to see that the isotropy subalgebra at $p$ preserves no 2-dimensional subspaces, and this contradiction finalizes the proof. Notice that this second argument does not rely on the sphericity (or even Levi-nondegeneracy) property of $M\setminus V$. In fact, combined with the second and third paragraphs of the proof of Theorem \ref{main1}, it shows that no 5-dimensional real-analytic CR-hypersurface admits a local action of either of the groups $\SU(1,3)$, $\SU(2,2)$ with an orbit of positive codimension.
\end{remark}

\section{Examples of CR-hypersurfaces with $\dim\mathfrak{s}=11$}\label{examples}
\setcounter{equation}{0}

We will now elaborate on the case $\dim\mathfrak{s}=11$ as stated in Theorem \ref{main1} and show that it is in fact realizable. More precisely, in our second main result below we will give a countable number of pairwise nonequivalent examples with $\dim\mathfrak{s}=11$. By Proposition \ref{su22}, if $\dim\mathfrak{s}=11$, the algebra $\mathfrak{s}$ is isomorphic to ${\mathfrak p}_2=\mathfrak{co}(1,3)\oright\R^{1,3}\subset{\mathfrak{su}}(2,2)$, and we start by collecting basic facts on this subalgebra.

\subsection{The subalgebra ${\mathfrak p}_2$}\label{algp2}

Realize $\mathfrak{su}(2,2)$ as 
$$
\mathfrak{su}(2,2)=\left\{X=\begin{pmatrix}A & B\\ C & -A^*\end{pmatrix}: A\in\mathfrak{gl}(2,\C), \op{Tr}(A)\in\R, B=B^*, C=C^*\right\},
 $$
where the lower-triangular, block-diagonal and upper-triangular parts give $\g_{-1}$, $\g_0$ and $\g_1$, respectively (cf.~Proposition \ref{su22} and Remark \ref{remp22}). In what follows we identify $\begin{pmatrix}A & 0\\ 0 & -A^*\end{pmatrix}$ with $A$, $\begin{pmatrix}0 & B\\ 0 & 0\end{pmatrix}$ with $B$, $\begin{pmatrix}0 & 0\\ C & 0\end{pmatrix}$ with $C$ and understand commutators among $A$, $B$, $C$ as those among the corresponding extended matrices. 

The action of $\g_0=\mathfrak{sl}(2,\C)_{\R}\oplus\R$ on $\g_{-1}$ is
\begin{equation}
[A,C]=-(A^*C+CA),\label{actionong-1}
\end{equation}
and, as was stated in the proof of Theorem \ref{main1} in Section \ref{first}, the induced action of $\mathfrak{sl}(2,\C)_{\R}={\mathfrak{so}}(1,3)$ is, up to an automorphism, its standard action on $\R^{1,3}$. This can be seen by writing any Hermitian matrix $C$ as
$$
C=\left(
\begin{array}{ll}
t+z & x-iy\\
x+iy & t-z
\end{array}
\right),\quad\hbox{with $t,x,y,z\in\R$,}
$$
and identifying $\g_{-1}$ with $\R^{1,3}$ by means of the 4-tuple $(t,x,y,z)$. 

Similarly, the action of $\g_0$ on $\g_1$ is $[A, B]=AB+BA^*$, which corresponds to the standard action of $\mathfrak{sl}(2,\C)_{\R}={\mathfrak{so}}(1,3)$ on $\R^{1,3}$. This fact yields structure relations for ${\mathfrak p}_2$ as shown below. 

Let $\{X_k, X_k^i, R\}_{k=1,2,3}$, and $\{V_{\ell}\}_{\ell=1,2,3,4}$ be the following bases in $\mathfrak{sl}(2,\C)_{\R}\oplus\R$ and the space of Hermitian $2\times 2$-matrices, respectively:
$$
\begin{array}{llll}
X_1:=\begin{pmatrix}0 & 0\\ 1 & 0\end{pmatrix},& X_1^i:=\begin{pmatrix}0 & 0\\ i & 0\end{pmatrix}, & X_2:=\begin{pmatrix}\frac12 & 0\\ 0 & -\frac12\end{pmatrix}, &
X_2^i:=\begin{pmatrix}\frac{i}2 & 0\\ 0 & -\frac{i}2\end{pmatrix},\\
\vspace{-0.1cm}\\
X_3:=\begin{pmatrix}0 & -1\\ 0 & 0\end{pmatrix}, & X_3^i:=\begin{pmatrix}0 & -i\\ 0 & 0\end{pmatrix}, & R:=\begin{pmatrix}\frac12 & 0\\ 0 & \frac12\end{pmatrix},&
V_1:=\begin{pmatrix}1 & 0\\ 0 & 0\end{pmatrix},\\
\vspace{-0.1cm}\\
V_2:=\begin{pmatrix}0 & 1\\ 1 & 0\end{pmatrix},&
V_3:=\begin{pmatrix}0 & i\\ -i & 0\end{pmatrix},&
V_4:=\begin{pmatrix}0 & 0\\ 0 & 1\end{pmatrix}.&
\end{array}
$$
The nontrivial commutators among the above elements are: for the basis $\{X_k\}_{k=1,2,3}$ of $\mathfrak{sl}(2,\R)\subset\mathfrak{sl}(2,\C)_{\R}$ we have $[X_1,X_2]=X_1$, $[X_1,X_3]=2X_2$, $[X_2,X_3]=X_3$ (and the commutators involving the superscript $i$ are the obvious consequences); next, we see that $[R, V_{\ell}]=V_{\ell}$ for all $\ell$; finally, for the representation of $\mathfrak{sl}(2,\C)_{\R}$ on $\g_1$ we calculate
\begin{equation}
\hspace{0.8cm}\makebox[250pt]{$\begin{array}{llll}
[X_1, V_1]=V_2,& [X_1,V_2]=2V_4,& [X_1^i, V_1]=-V_3,& [X_1^i, V_3]=-2V_4,\\
\vspace{-0.3cm}\\

[X_2, V_1] =V_1,& [X_2, V_4]=-V_4,& [X_2^i, V_2]=V_3,& [X_2^i, V_3]=-V_2,\\
\vspace{-0.3cm}\\

[X_3, V_2] =-2V_1,& [X_3, V_4]=-V_2,& [X_3^i, V_3]=-2V_1,& [X_3^i, V_4]=-V_3.
\end{array}$}\label{relations8888}
\end{equation}

\subsection{The examples}\label{main21}

Denote by $z=x+iy$, $w=u+iv$, $t=\tau+i\sigma$ the  coordinates in $\C^3$ and for every $n=1,2,\dots$, let $M_n\subset\C^3$ be the hypersurface defined for $-\pi/2<v<\pi/2$ by the equation
\begin{equation}
\sigma=\tau\tan\left(\frac{1}{n}\tan^{-1}\frac{e^u\sin v-2y}{e^u\cos v}\right).\label{eqexamp}
\end{equation}
Notice that $M_n$ is Levi-degenerate precisely at the points of the complex hypersurface ${\mathfrak S}:=\{-\pi/2<v<\pi/2,t=0\}=M_n\cap\{\tau=0\}$ and therefore, by Theorem \ref{main1}, we have $\dim\mathfrak{s}\le 11$. Clearly, $M_n$ is not minimal, hence not of finite type (in the sense of Kohn and Bloom-Graham), at any point of ${\mathfrak S}$ (see \cite[\S 1.5]{BER}).

The complement $M_n\setminus{\mathfrak S}$ has exactly two connected components. They are defined by the sign of $\tau$ and we call them $M_n^{+}$ and $M_n^{-}$, respectively. A short computation shows that $M_n\setminus{\mathfrak S}$ is given by
$$
\Im\left(\frac{e^w}{t^n}\right)=2\Re\left(\frac{1}{t^n}\right)y,\quad t\ne 0.
$$
It is then easy to see that near every point in $M_n\setminus{\mathfrak S}$ the holomorphic map
$$
z\mapsto -i z,\quad w\mapsto -\frac{ie^w}{2t^n},\quad t\mapsto\frac{1}{t^n}
$$
is a CR-diffeomorphism onto an open subset of the quadric
$$
\Re(w)=\Re(z)\Re(t)
$$
(cf.~(\ref{QK})). Hence, $M_n\setminus{\mathfrak S}$ is spherical with Levi form of signature $(1,1)$.

We are now ready to state and prove the second main theorem of the paper.  

\begin{theorem}\label{main2} For every $n=1,2,\dots$ the symmetry algebra of $M_n$ is {\rm 11}-dimensional; in fact one has $\mathfrak{s}=\mathfrak{p}_2$. Furthermore, if $p_1, p_2\in {\mathfrak S}$, then for $n\ne k$ the germs of the hypersurfaces $M_n$, $M_k$ at $p_1$, $p_2$, respectively, are not equivalent by means of a smooth CR-diffeomorphism. In particular, $M_n$, $M_k$ are not CR-equivalent even as smooth CR-manifolds.
\end{theorem}

\begin{proof} First, for every $n$ we explicitly write a faithful representation by holomorphic vector fields on $\C^3$ of the algebra $\mathfrak{p}_2$, where for $X\in\mathfrak{p}_2$ the corresponding vector field is denoted by $\vf(X)$:   
\begin{equation}
\begin{array}{l}
\displaystyle\vf(X_1):=\frac{\partial}{\partial z},\quad
\displaystyle \vf(X_2):=z\frac{\partial}{\partial z}+\frac{\partial}{\partial w},\\
\vspace{-0.3cm}\\ 
\displaystyle \vf(X_3):=z^2\frac{\partial}{\partial z}+(2z-e^w)\frac{\partial}{\partial w}-e^w\frac{t}{n} \frac{\partial}{\partial t},\quad 
\displaystyle \vf(X_1^i):=i\frac{\partial}{\partial z}+2i e^{-w}\frac{\partial}{\partial w},\\
\vspace{-0.3cm}\\ 
\displaystyle \vf(X_2^i):=iz\frac{\partial}{\partial z}+i(2z e^{-w}-1)\frac{\partial}{\partial w}-i\frac{t}{n}\frac{\partial}{\partial t},\\ 
\vspace{-0.3cm}\\ 
\displaystyle \vf(X_3^i):=iz^2 \frac{\partial}{\partial z}+i(e^w+2z^2 e^{-w}-2z)\frac{\partial}{\partial w}+i(e^w-2z)\frac{t}{n}\frac{\partial}{\partial t},\\
\vspace{-0.3cm}\\ 
\displaystyle \vf(R):=\frac{t}{n}\frac{\partial}{\partial t},\quad
\displaystyle \vf(V_1):=zt^n(ze^{-w}-1)\frac{\partial}{\partial w}-z\frac{t^{n+1}}{n}\frac{\partial}{\partial t},\\
\vspace{-0.3cm}\\ 
\displaystyle \vf(V_2):=t^n(2ze^{-w}-1)\frac{\partial}{\partial w}-\frac{t^{n+1}}{n}\frac{\partial}{\partial t},\\
\vspace{-0.3cm}\\ 
\displaystyle \vf(V_3):=it^n\frac{\partial}{\partial w}+i\frac{t^{n+1}}{n}\frac{\partial}{\partial t},\quad
\displaystyle \vf(V_4):= t^n e^{-w} \frac{\partial}{\partial w}. 
\end{array}\label{repres}
\end{equation}
A straightforward (albeit tedious) calculation now shows that the real parts of the holomorphic vector fields in (\ref{repres}) are indeed tangent to $M_n$, hence $\mathfrak{s}=\mathfrak{p}_2$. Similarly, for any connected neighborhood ${\mathcal U}$ in $M_n$ of a point $p\in{\mathfrak S}$ we have $\hol({\mathcal U})=\mathfrak{p}_2$. 

To prove the second statement of the theorem we need the following lemma.

\begin{lemma}\label{analyticity}
For all $n$ and $p\in{\mathfrak S}$ we have $\mathfrak{aut}(M_n,p)=\hol(M_n)$.
\end{lemma}

\begin{proof} Fix a neighborhood ${\mathcal U}$ of $p$ in $M$ such that ${\mathcal U}^{+}:={\mathcal U}\cap M_n^{+}$ and ${\mathcal U}^{-}:={\mathcal U}\cap M_n^{-}$ are connected. As ${\mathcal U}^{\pm}$ is spherical with Levi form of signature $(1,1)$ and $M_n^{\pm}$ is simply-connected, we have $\mathfrak{aut}({\mathcal U}^{\pm})=\hol({\mathcal U}^{\pm})=\mathfrak{su}(2,2)$. On the other hand, we have $\hol({\mathcal U})=\mathfrak{p}_2$. Hence, there is a copy of $\mathfrak{p}_2$ in $\mathfrak{su}(2,2)$ (conjugate by some  element $g^{\pm}\in\SU(2,2)$ to the copy given by upper-triangular matrices as in Section \ref{algp2}) that under the isomorphism $\mathfrak{su}(2,2)\cong\hol({\mathcal U}^{\pm})$ is mapped into the subalgebra $\mathfrak{a}^{\pm}\subset\hol({\mathcal U}^{\pm})$ spanned by the holomorphic vector fields in the right-hand side of (\ref{repres}) restricted to either $\{\tau>0\}$ or $\{\tau<0\}$, respectively. Under this isomorphism, by relation (\ref{actionong-1}) every element $U\in g^{\pm}\,\mathfrak{g}_{-1}\,(g^{\pm})^{-1}$ (where $\mathfrak{g}_{-1}$ is given by lower-triangular matrices) is mapped to a holomorphic vector field $X_U$ defined near ${\mathcal U}^{\pm}$ such that
$$
[\vf(R),X_U]=-X_U,
$$
which implies that $X_U$ has the form
\begin{equation}
X_{U}=t^{-n}A(z,w)\frac{\partial}{\partial z}+ t^{-n}B(z,w)\frac{\partial}{\partial w}+ t^{-(n+1)}C(z,w)\frac{\partial}{\partial t},\label{specform1}
\end{equation}
where $A$, $B$, $C$ are holomorphic functions.

Now, fix a vector field $X$ representing an element of $\mathfrak{aut}(M_n,p)$ and show that $X\in\hol(M_n)$. Choose a neighborhood ${\mathcal U}$ of $p$ in $M$ as above in which $X$ is defined and consider the restrictions $X^{\pm}:=X|_{{\mathcal U}^{\pm}}\in\mathfrak{h}({\mathcal U}^{\pm})$. Then $X^{\pm}$ is the sum of an element in $\mathfrak{a}^{\pm}$ and the real part of a vector field of the form (\ref{specform1}) restricted to ${\mathcal U}^{\pm}$. As $X$ is smooth, we then see that $X^{\pm}\in\mathfrak{a}^{\pm}$, and it follows from (\ref{repres}) that $X^{+}$ and $X^{-}$ glue together into an element of $\hol(M_n)$. Hence $\mathfrak{aut}(M_n,p)=\hol(M_n)$ as required.\end{proof}

We will now prove the second statement of the theorem. First of all, it follows from (\ref{repres}) that $\SL(2,\C)_{\R}$ acts real-analytically on $M_n$ for every $n$ and that ${\mathfrak S}$ is an orbit of this action. In fact, it is not hard to see that the real parts of the vector fields $\vf(X_1)$, $\vf(X_2)$, $\vf(X_1^i)$, $\vf(X_2^i)$ generate 1-parameter subgroups of the group of (global) real-analytic CR-automorphisms of $M_n$ and for every pair of points in ${\mathfrak S}$ there exists a composition of elements of these subgroups mapping one point into the other. Therefore, it suffices to show that the germs of $M_n$, $M_k$ at the origin are not smoothly CR-equivalent if $n\ne k$.  

Assume the opposite and let $F:(M_n,0)\to (M_k,0)$ be the germ of a smooth CR-isomorphism at the origin. We then obtain an isomorphism between the Lie algebras $\mathfrak{aut}(M_n,0)$, $\mathfrak{aut}(M_k,0)$, hence, by Lemma \ref{analyticity}, an isomorphism between $\hol(M_n)$, $\hol(M_k)$, and therefore an automorphism of $\mathfrak{p}_2$. Notice that $\left<V_1,V_2,V_3,V_4\right>$ is the commutant of the radical subalgebra $\left<R,V_1,V_2,V_3,V_4\right>$ of ${\mathfrak p}_2$, which implies that any automorphism of ${\mathfrak p}_2$ takes $R$ to itself. It then follows that  
$$
F^*\left(2\Re\left(\frac{t}{k}\frac{\partial}{\partial t}\right)\Big|_{M_k}\right)=2\Re\left(\frac{t}{n}\frac{\partial}{\partial t}\right)\Big|_{M_n},
$$
where the vector fields are identified with their germs at the origin. This is, however, impossible as the spectrum of $2\Re\left(\frac{t}{k}\frac{\partial}{\partial t}\right)|_{M_k}$ at the origin is $(0,0,0,0,1/k)$ and must be preserved by $F^*$ up to a permutation. This completes the proof.\end{proof}

From Theorems \ref{main1}, \ref{main2} we now obtain the value of the submaximal dimension of the symmetry algebra $\mathfrak{s}$:

\begin{corollary}\label{smaxdim}
For the class of real-analytic connected holomorphically nondegenerate {\rm 5}-dimensional CR-hypersurfaces one has $D_{\hbox{\tiny\rm smax}}=11$.
\end{corollary}

\noindent Next, recalling the result of \cite{KS} stated in the introduction (see also \cite{IK}), we observe that it has the following analogue for $n=2$:

\begin{corollary}\label{dimconjnew} For a real-analytic connected holomorphically nondegenerate {\rm 5}-dimensional CR-hypersurface $M$ and a point $p\in M$, the condition $\dim\hol(M,p)>11$ implies that $M$ is spherical near $p$, and this estimate is sharp.
\end{corollary}

\begin{proof} We only need to prove the sharpness of the estimate $\dim\hol(M,p)>11$, which is a consequence of the fact, noted in the proof of Theorem \ref{main2}, that for any point $p\in{\mathfrak S}$ and all $n$ one has $\hol(M_n,p)=\mathfrak{p}_2$. 
\end{proof}

\section{Background for Theorem \ref{main2}}\label{fourdimorbit}
\setcounter{equation}{0}

In this section we explain how the examples from Theorem \ref{main2} were constructed. They are not a product of mere guesswork; rather, we searched for such examples in a systematic way. Our search strategy is summarized in Theorem \ref{examplegeneral} stated at the end of the section, which is a result of independent interest. 

Suppose that $M$ is a CR-hypersurface with $\mathfrak{s}=\mathfrak{p}_2$ as in Theorem \ref{main1}. We looked for potential examples assuming the existence of a local $P$-orbit of positive codimension in $M$, where the group $P:=(\SL(2,\C)_{\R}\times\RR)\ltimes\R^4$ has $\mathfrak{p}_2$ as its Lie algebra. For instance, such an orbit exists if $M$ is simply-connected. Indeed, otherwise $M$ would be spherical with Levi form of signature $(1,1)$, and the simply-connectedness of $M$ would then imply that $\mathfrak{s}=\mathfrak{su}(2,2)$.

Fix a point $p_0$ in a positive-codimensional $P$-orbit. Our arguments are based on considering the local orbit of $p_0$, say $\Sigma$, under the induced local action of the group $\SL(2,\C)_{\R}\subset P$. Since $\Sigma$ has positive codimension in $M$ as well, the isotropy subalgebra of $p_0$ under the $\SL(2,\C)_{\R}$-action has dimension at least 2. As we will see in Proposition \ref{complexstructure} at the end of this section, $\Sigma$ is far from being arbitrary; in fact it is either a complex curve or a complex surface in $M$. 

\subsection{Subalgebras of $\mathfrak{sl}(2,\C)_{\R}$ and local $\SL(2,\C)_{\R}$-orbits}\label{subalgorbitsclass}

\begin{proposition}\label{subalgebras}
Every subalgebra of $\mathfrak{sl}(2,\C)_{\R}$ of dimension $2\leq d<6$ is conjugate to one of 
\begin{itemize}
\item[$d=4$:] $\mathfrak{b}_{\R}$ where $\mathfrak{b}$ is the Borel subalgebra $\left\{\left(\begin{array}{cc} b & a\\ 0 & -b\end{array}\right): a,b\in\C\right\}$;\vspace{0.1cm}\\
\item[$d=3$:] $\mathfrak{sl}(2,\R)$, $\mathfrak{su}(2)$, and
$\R^1(\phi)\oright\R^2=
\left\{\left(\begin{array}{cc} b\,e^{i\phi} & a\\ 0 & -b\,e^{i\phi}\end{array}\right): a\in\C,b\in\R\right\}
\subset\mathfrak{b}_{\R}$ for some $\phi\in\R$;
\vspace{0.1cm}\\
\item[$d=2$:] $\mathfrak{c}_{\R}$ where $\mathfrak{c}$ is the Cartan subalgebra $\{\op{diag}(a,-a):a\in\C\}$, $\mathfrak{n}_{\R}$ where\linebreak $\mathfrak{n}$ is the Abelian subalgebra $\left\{\left(\begin{array}{cc} 0 & a\\ 0 & 0\end{array}\right): a\in\C\right\}$, and the solvable
subalgebra ${\mathfrak{sol}}_2:=\left\langle \left(\begin{array}{cc} 1 & 0\\ 0 & -1\end{array}\right),
\left(\begin{array}{cc} 0 & 1\\ 0 & 0\end{array}\right)\right\rangle$.
\end{itemize}
\end{proposition}

\begin{proof}
As in the proofs of Propositions \ref{su13}, \ref{su22}, we again use Mostow's theorem to list the maximal subalgebras of $\mathfrak{sl}(2,\C)_{\R}$. The only parabolic subalgebra, up to conjugation, is $\mathfrak{b}_{\R}$. Its every codimension 1 subalgebra is conjugate to $\R^1(\phi)\oright\R^2$ for some $\phi\in\R$. Furthermore, in dimension 3 one also has two simple
subalgebras of $\mathfrak{sl}(2,\C)_{\R}$, namely $\mathfrak{sl}(2,\R)$ and $\mathfrak{su}(2)$. Next, any pseudotoric subalgebra of $\mathfrak{sl}(2,\C)_{\R}$ is conjugate to $\mathfrak{c}_{\R}$. The other 2-dimensional subalgebras are given by specializing the number of semisimple generators: 1 for $\mathfrak{sol}_2$ and 0 for $\mathfrak{n}_{\R}$. \end{proof}
 
\begin{remark}\label{classification} The simple classification in Proposition \ref{subalgebras} has been rediscovered several times, especially when describing subalgebras of the Lorentz algebra $\mathfrak{so}(1,3)$ (see, e.g., \cite{PWZ} and references therein), and we only provide it here for the completeness of our exposition. To augment the above classification, we also note that all 1-dimensional subalgebras of $\mathfrak{sl}(2,\C)_{\R}$ are specified by Jordan normal forms.
 \end{remark}
 
As the following proposition shows, not all the subalgebras from Proposition \ref{subalgebras} are realizable as isotropy subalgebras in the case at hand.
 
\begin{proposition}\label{isotropysubalgebras}
The isotropy subalgebra $\h\subset\mathfrak{sl}(2,\C)_{\R}$ of the point $p_0\in\Sigma$ is conjugate to one of $\mathfrak{b}_{\R}$, $\mathfrak{c}_{\R}$, $\mathfrak{n}_{\R}$.
\end{proposition}
 
\begin{proof} Let $d:=\dim\h$. For $d=6$ we have $\Sigma=\{p_0\}$, and this case is ruled out by linearizing the action of $\h=\mathfrak{sl}(2,\C)_{\R}$ and arguing as in the last paragraph in the proof of Theorem \ref{main1}.

Let $d=3$. In this case the local $\SL(2,\C)_{\R}$-orbit $\Sigma$ is 3-dimensional and has an invariant complex line distribution, namely, $T(\Sigma)\cap JT(\Sigma)$. If $\h$ is conjugate to one of the simple algebras $\mathfrak{sl}(2,\R)$, $\mathfrak{su}(2)$, then we have $\mathfrak{sl}(2,\C)_{\R}=\mathfrak{h}+i\,\mathfrak{h}$, and therefore the isotropy representation of $\h$ is irreducible contradicting the existence of an invariant rank 2 distribution. If $\h$ is conjugate to $\R^1(\phi)\oright\R^2$ for some $\phi$, then
the isotropy representation of $\h$ is reducible but not decomposable: there exists an invariant 
1-dimensional subspace but no invariant 2-dimensional subspace, which again leads to a contradiction. This rules out the case $d=3$. 

Suppose finally that $d=2$ and $\mathfrak{h}$ is conjugate to $\mathfrak{sol}_2$. In this situation, from the isotropy representation of $\h$ one immediately observes that $\Sigma$ is not complex, so as in Remark \ref{anotherproof} we let $L:=H(M)\cap T(\Sigma)$ and consider the invariant complex line distribution $L\cap JL$. However, it is easy to see that the isotropy representation of $\h$ has no invariant 2-dimensional subspaces. This completes the proof of the proposition. \end{proof}
 
For future reference, we also state the following lemma, which is obtained by a direct elementary analysis of the isotropy representation of $\h$ as above:

\begin{lemma}\label{cartancomplexorbit} Let $\h\subset\mathfrak{sl}(2,\C)_{\R}$ be the isotropy subalgebra of the point $p_0\in\Sigma$. If $\h$ is conjugate to $\mathfrak{c}_{\R}$, then $\Sigma$ is a complex surface in $M$. If $\h$ is conjugate to $\mathfrak{b}_{\R}$, then $\Sigma$ is either a complex curve or a totally real surface in $M$.
\end{lemma}

\subsection{The case of Cartan subalgebras}\label{caseCartan}

We will now consider in detail the situation when the isotropy subalgebra of $p_0$ under the $\SL(2,\C)_{\R}$-action is (the realification of) a Cartan subalgebra of $\mathfrak{sl}(2,\C)$. Among the three possibilities listed in Proposition \ref{isotropysubalgebras}, this is perhaps the most interesting one. In fact, as we will see in Theorem \ref{examplegeneral} below, the examples of Theorem \ref{main2} arise from this case, and the arguments presented here explain how exactly we arrived at formulas (\ref{eqexamp}), (\ref{repres}).

Without loss of generality we may suppose that the isotropy subalgebra, say $\tilde {\mathfrak c}_{\R}$, is spanned by $Y_0:=X_1+X_2$ and $Y_0^i:=X_1^i+X_2^i$ (here and below we utilize the notation and commutation relations from Section \ref{algp2} without explicit reference). By Lemma \ref{cartancomplexorbit}, the orbit $\Sigma$ is a complex surface in $M$. As $\Sigma$ is 4-dimensional, it coincides with the local $P$-orbit of $p_0$.

Consider the following Borel subgroup of $\SL(2,\C)$: 
$$
B:=\left\{\left(\begin{array}{cc} e^{b/2} & a\\ 0 & e^{-b/2}\end{array}\right)\Biggl\rvert\, a,b\in\C\right\}.
$$
Clearly, $B_{\R}$ acts simply transitively on $\Sigma$ near $p_0$. Set ${\mathbf z}:=-ae^{b/2}$ and ${\mathbf w}:=b$. The real and imaginary parts of ${\mathbf z}, {\mathbf w}$ form a local real coordinate system on $\Sigma$ centered at $p_0$. Furthermore, the pair $({\mathbf z},{\mathbf w})$ defines an $\SL(2,\C)_{\R}$-invariant complex structure on $\Sigma$, which is exactly the complex structure that comes from the quotient $\SL(2,\C)/\tilde C$ of complex Lie groups, where $\tilde C$ is the Cartan subgroup with\linebreak Lie algebra $\tilde {\mathfrak c}$. 

Let, as before, $\mathfrak{b}$ be the Lie algebra of $B$ and consider the following two elements in it: $Y_1:=-X_3$ and $Y_2:=-X_2$. With respect to the complex structure defined by $({\mathbf z},{\mathbf w})$, for the fundamental holomorphic vector fields $\hat Y_0$, $\hat Y_1$, $\hat Y_2$ on $\Sigma$ arising from $Y_0$, $Y_1$, $Y_2$, respectively, we have
\begin{equation}
\begin{array}{l}
\displaystyle\hat Y_0=\left(-{\mathbf z}-{\mathbf z}^2\right)\frac{\partial}{\partial {\mathbf z}}+\left(-1-2{\mathbf z}+e^{\mathbf w}\right)\frac{\partial}{\partial {\mathbf w}},\\
\vspace{-0.3cm}\\
\displaystyle\hat Y_1=\frac{\partial}{\partial {\mathbf z}},\quad \hat Y_2={\mathbf z}\frac{\partial}{\partial {\mathbf z}}+\frac{\partial}{\partial {\mathbf w}}.
\end{array}\label{vfstand}
\end{equation}

Identify $T_{p_{{}_0}}(\Sigma)\cong\mathfrak{sl}(2,\C)_{\R}/\tilde {\mathfrak c}_{\R}\cong\mathfrak{b}_{\R}$. It is then easy to find all $\tilde {\mathfrak c}_{\R}$-invariant complex structures on $T_{p_{{}_0}}(\Sigma)$. In the complex coordinates defined on $T_{p_{{}_0}}(\Sigma)$ by $\partial/\partial {\mathbf z}|_{(0,0)}$, $\partial/\partial {\mathbf w}|_{(0,0)}$ any such structure is given by one of the matrices
$$
\pm\left(\begin{array}{ll} i & 0\\ 0 & i\end{array}\right),\quad \pm\left(\begin{array}{rl} i & 0\\ 2i & -i\end{array}\right),
$$
each of which leads to an integrable $\SL(2,\C)_{\R}$-invariant almost complex structure on $\Sigma$. Namely, at a point $(z,w)$ we obtain, respectively,
\begin{equation}
J_1^{\pm}({\mathbf z},{\mathbf w}):=\pm\left(\begin{array}{ll} i & 0\\ 0 & i\end{array}\right),\quad J_2^{\pm}({\mathbf z},{\mathbf w}):=\pm\left(\begin{array}{cc} i & 0\\ 2ie^{-{\mathbf w}} & -i\end{array}\right).\label{complstructures}
\end{equation}
For each $k$ the complex structures arising from $J_k^{+}$, $J_k^{-}$ are conjugate to each other and all considerations for them are identical. Therefore, in what follows we will focus on $J_1^{+}$ and $J_2^{+}$. Notice that $J_1^{+}$ is the structure induced by the local complex coordinates $({\mathbf z},{\mathbf w})$ as discussed above. It is also not hard to check that $J_2^{+}$ is induced by the local complex coordinates $({\mathbf z}^*,{\mathbf w}^*):=({\mathbf z}, \ln({\mathbf z}-\bar{{\mathbf z}}+e^{\overline {\mathbf w}}))$.

Let $\widehat {Y_0}$, $\widehat {Y_1}$, $\widehat {Y_2}$, $\widehat {Y_0^i}$, $\widehat {Y_1^i}$, $\widehat {Y_2^i}$ be the fundamental holomorphic vector fields with respect to the structure $J_2^{+}$ on $\Sigma$ arising from $Y_0$, $Y_1$, $Y_2$, $Y_0^i$, $Y_1^i:=-X_3^i$, $Y_2^i:=-X_2^i$, respectively. Clearly, $\hbox{Re}\,\hat Y_j=\hbox{Re}\,\widehat {Y_j}$ for $j=0,1,2$. In the coordinates $({\mathbf z}^*,{\mathbf w}^*)$ these vector fields are written as follows:
\begin{equation}
\hspace{0.8cm}\makebox[250pt]{$\begin{array}{l}
\displaystyle\widehat{Y_0}=-\left({\mathbf z}^*+{\mathbf z}^{*2}\right)\frac{\partial}{\partial {\mathbf z}^*}+\left(-1-2{\mathbf z}^*+e^{{\mathbf w}^*}\right)\frac{\partial}{\partial {\mathbf w}^*},\\
\vspace{-0.3cm}\\
\displaystyle\widehat{Y_1}=\frac{\partial}{\partial {\mathbf z}^*},\quad \widehat Y_2={\mathbf z}^*\frac{\partial}{\partial {\mathbf z}^*}+\frac{\partial}{\partial {\mathbf w}^*},\\
\vspace{-0.3cm}\\
\displaystyle \widehat {Y_0^i}=-i\left({\mathbf z}^*+{\mathbf z}^{*2}\right)\frac{\partial}{\partial {\mathbf z}^*}-i\left(-1-2{\mathbf z}^*+e^{{\mathbf w}^*}+2{\mathbf z}^*e^{-{\mathbf w}^*}+2{\mathbf z}^{*2}e^{-{\mathbf w}^*}\right)\frac{\partial}{\partial {\mathbf w}^*},\\
\vspace{-0.3cm}\\
\displaystyle \widehat {Y_1^i}=i\frac{\partial}{\partial {\mathbf z}^*}+2ie^{-{\mathbf w}^*}\frac{\partial}{\partial {\mathbf w}^*},\quad \widehat {Y_2^i}=i{\mathbf z}^*\frac{\partial}{\partial {\mathbf z}^*}-i(1-2{\mathbf z}^*e^{-{\mathbf w}^*})\frac{\partial}{\partial {\mathbf w}^*}.
\end{array}$}\label{vfstand1}
\end{equation}

Recall now that the manifold $M$ can be regarded as a closed hypersurface in a 3-dimensional complex manifold ${\mathcal M}$. We will write the effective action of $\mathfrak{p}_2$ on $M$ as a monomorphism into the algebra of holomorphic vector fields defined in some fixed neighborhood of $M$ in ${\mathcal M}$. As above, we will use the notation $\vf(X)$ to denote the image of $X\in\mathfrak{p}_2$. 

Everywhere below all holomorphic vector fields will be written in local holomorphic coordinates $(z,w,t)$ in ${\mathcal M}$ centered at $p_0$. The coordinates can be chosen to satisfy the following two conditions. First of all, we require that $\Sigma=\{t=0\}$. Secondly, notice that since each pair of vectors $\hat Y_1(p_0)$, $\hat Y_2(p_0)$ and $\widehat{Y_1}(p_0)$, $\widehat {Y_2}(p_0)$ is complex-linearly independent and $[Y_1,Y_2]=Y_1$, one can pick $(z,w,t)$ so that
\begin{equation}
\vf(Y_1)=\frac{\partial}{\partial z},\quad \vf(Y_2)=z\frac{\partial}{\partial z}+\frac{\partial}{\partial w}\label{X1X2}
\end{equation}
(cf.~(\ref{vfstand}), (\ref{vfstand1})). 

For the structure $J_1^{+}$ one then has $(z,w)|_{\Sigma}=({\mathbf z},{\mathbf w})$ and $\vf(Y_j)|_{\Sigma}=\hat Y_j$, whereas for the structure $J_2^{+}$ one has $(z,w)|_{\Sigma}=({\mathbf z}^*,{\mathbf w}^*)$ and $\vf(Y_j)|_{\Sigma}=\widehat {Y_j}$, $j=1,2$. Furthermore, as $[Y_0,Y_1]=Y_1+2Y_2$ and $[Y_0,Y_2]=-Y_0-Y_2$, we obtain
\begin{equation}
\begin{array}{l}
\displaystyle\vf(Y_0)=\left(-z-z^2+f_1(t)e^{2w}\right)\frac{\partial}{\partial z}\\
\vspace{-0.3cm}\\
\hspace{4cm}\displaystyle+\left(-1-2z+f_2(t)e^w\right)\frac{\partial}{\partial w}+f_3(t)e^w\frac{\partial}{\partial t},
\end{array}\label{rep0}
\end{equation}
where $f_j$ are holomorphic and satisfy $f_1(0)=f_3(0)=0$, $f_2(0)=1$ (cf.~(\ref{vfstand}), (\ref{vfstand1})). We stress here that the form of the vector fields $\vf(Y_0^i)$, $\vf(Y_1^i)$, $\vf(Y_2^i)$ is determined by picking one of the two almost complex structures $J_1^{+}$, $J_2^{+}$  (cf.~(\ref{vfstand1})). 

We will now consider five cases depending on whether or not the value $f_3'(0)$ and/or the function $f_3$ is nonzero as well as on the choice of the almost complex structure $J_k^{+}$.
\vspace{-0.1cm}\\

{\bf Case 1.} Suppose that $f_3'(0)=0$ but $f_3\ne 0$. In this situation, we will make a formal (possibly divergent) change of variables to ensure that in formula (\ref{rep0}) one has $f_1=0$, $f_2=1$. The fact that the change of variables is only formal is not going to affect our arguments below. Indeed, we will use certain commutation relations to show that some vector field arising from the action of $\mathfrak{p}_2$ is zero and thus obtain a contradiction with the effectivity of the action. It is clear that this conclusion is independent of formal changes of variables.

Let us transform the variables $(z,w,t)$ as follows:
\begin{equation}
z\mapsto z+F(t)e^w,\quad w\mapsto w+G(t),\quad t\mapsto t,\label{change88}
\end{equation}
where $F$ and $G$ are formal power series with vanishing constant term. It is easy to check that such a transformation preserves the form (\ref{X1X2}) of $\vf(Y_1)$, $\vf(Y_2)$. Hence, it also preserves the form (\ref{rep0}) of $\vf(Y_0)$, and the functions $f_j$ can be shown to change as
\begin{equation}
\begin{array}{l}
f_1\mapsto f_1^*:=(f_1+Ff_2+F'f_3+F^2)e^{-2G},\\
\vspace{-0.3cm}\\
f_2\mapsto f_2^*:=(f_2+G'f_3+2F)e^{-G},\\
\vspace{-0.3cm}\\
f_3\mapsto f_3^*:=f_3e^{-G}.
\end{array}\label{change888}
\end{equation}
Since $f_3'(0)=0$, it is clear that one can find a formal power series $F$ for which $f_1^*=0$. With $F$ chosen in this way, one can analogously determine a formal power series $G$ that insures $f_2^*=1$. Thus, setting $f:=f_3^*$, we write $\vf(Y_0)$ as
\begin{equation}
\vf(Y_0)=\left(-z-z^2\right)\frac{\partial}{\partial z}+\left(-1-2z+e^w\right)\frac{\partial}{\partial w}+f(t)e^w\frac{\partial}{\partial t},\label{rep01}
\end{equation}
where $f'(0)=0$ (cf.~(\ref{rep0})).

As $[R,Y_1]=0$, $[R,Y_2]=0$, it follows from (\ref{X1X2}) that
\begin{equation}
\vf(R)=a_1(t)e^w\frac{\partial}{\partial z}+b_1(t)\frac{\partial}{\partial w}+c_1(t)\frac{\partial}{\partial t},\label{formRR}
\end{equation}
and the identity $[R,Y_0]=0$ together with formula (\ref{rep01}) yields
\begin{equation}
a_1+fa_1'=0,\quad b_1-fb_1'-2a_1=0,\quad f'c_1-fc_1'+fb_1=0.\label{rcommx0}
\end{equation}
Since $f'(0)=0$, equations (\ref{rcommx0}) together with power series decomposition imply
\begin{equation}
a_1=0,\quad b_1=0,\quad c_1=kf, \label{a1b1c1}
\end{equation}
where $k\in\C\setminus\{0\}$. 

Next, as $[Y_1,V_1]=0$ and $[Y_2,V_1]=-V_1$, we see
\begin{equation}
\vf(V_1)=a_2(t)\frac{\partial}{\partial z}+b_2(t)e^{-w}\frac{\partial}{\partial w}+c_2(t)e^{-w}\frac{\partial}{\partial t}.\label{vfV}
\end{equation}
The identity $[R,V_1]=V_1$ together with (\ref{formRR}), (\ref{a1b1c1}) now leads to
$$
a_2-kfa_2'=0,\quad b_2-kfb_2'=0,\quad c_2-kfc_2'+kf'c_2=0.
$$
Since $f'(0)=0$, analyzing the above equations analogously to (\ref{rcommx0}), we obtain $a_2=0$, $b_2=0$, $c_2=0$. Hence $\vf(V_1)=0$, which is impossible since the $\mathfrak{p}_2$-action is effective.
\vspace{0.1cm}\\

{\bf Case 2.} Suppose next that in (\ref{rep0}) we have $f_3=0$ and the almost complex structure induced on $\Sigma$ by $M$ is $J_1^{+}$ (see (\ref{complstructures})). By solving the system
$$
\begin{array}{l}
f_1+Ff_2+F^2=0,\\
\vspace{-0.3cm}\\
f_2+2F=e^G
\end{array}
$$ 
for $F$ and $G$ (see (\ref{change88}), (\ref{change888})), we can assume that in (\ref{rep0}) one also has $f_1=0$, $f_2=1$, i.e., 
\begin{equation}
\vf(Y_0)=\left(-z-z^2\right)\frac{\partial}{\partial z}+\left(-1-2z+e^w\right)\frac{\partial}{\partial w}.\label{rep02}
\end{equation}

Next, as in Case 1, we see that $\vf(R)$, $\vf(V_1)$ have the forms (\ref{formRR}), (\ref{vfV}), respectively. Furthermore, the identity $[R,Y_0]=0$ together with formula (\ref{rep02}) yields $a_1=0$, $b_1=0$ (cf.~(\ref{a1b1c1})). Then the identity $[R,V_1]=V_1$ leads to
\begin{equation}
a_2-c_1a_2'=0,\quad b_2-c_1b_2'=0,\quad c_2-c_1c_2'+c_1'c_2=0.\label{rel333}
\end{equation}  

Let $\beta:=c_1'(0)$. As $c_2(0)=0$, relations (\ref{rel333}) together with power series decomposition immediately imply that if $\beta\ne 1/n$ for any $n\in\N$, then $\vf(V_1)=0$, which contradicts the effectiveness of the $\mathfrak{p}_2$-action. Therefore, $\beta=1/n$ for some $n\in\N$, and by changing the variable $t$ one can assume that
\begin{equation}
R=\frac{t}{n}\frac{\partial}{\partial t}.\label{formRRR}
\end{equation}
For $\beta=1/n$ from (\ref{rel333}) we deduce
\begin{equation}
a_2=qt^n+o(t^n), \quad b_2=rt^n+o(t^n),\quad c_2=st^{n+1}+o(t^{n+1})\label{rel41}
\end{equation}
where $q,r,s\in\C$. Notice that the condition $q=r=s=0$ implies $\vf(V_1)=0$, thus at least one of $q,r,s$ is nonzero.

We will now utilize a number of commutation relations among vector fields arising from the action of $\mathfrak{sl}(2,\C)_{\R}$. In all such relations we will only focus on the coefficients at $t\partial/\partial t$, which will yield constraints on some of the functional parameters involved. Write $\vf(Y_2^i)$ in general form as
\begin{equation}
\vf(Y_2^i)=(iz+O(t))\frac{\partial}{\partial z}+(i+O(t))\frac{\partial}{\partial w}+(tc_1(z,w)+o(t))\frac{\partial}{\partial t}\footnote{Here and below $O(t^k)$ (resp.~$o(t^k)$) denotes any function of the form $\sum_{j=k}^{\infty}a_j(z,w)t^j$ (resp.~$\sum_{j=k+1}^{\infty}a_j(z,w)t^j$), where $a_j(z,w)$ is holomorphic.}\label{rep22}
\end{equation}
(cf.~(\ref{X1X2})). Then, since $Y_1^i=[Y_1,Y_2^i]$, we obtain
\begin{equation}
\vf(Y_1^i)=(i+O(t))\frac{\partial}{\partial z}+O(t)\frac{\partial}{\partial w}+\left(t\frac{\partial c_1}{\partial z}+o(t)\right)\frac{\partial}{\partial t}\label{rep11}
\end{equation}   
(cf.~(\ref{X1X2})). Next, from the identities $[Y_1,Y_1^i]=0$ and $[Y_1^i,Y_2]=Y_1^i$ we see
\begin{equation}
c_1=\tilde c_1 ze^{-w}+\check c_1(w),\label{formc1}
\end{equation}                   
with $\tilde c_1\in\C$. The relation $[Y_1^i,Y_2^i]=-Y_1$ together with (\ref{rep22}), (\ref{rep11}) now yields $\tilde c_1=0$. Hence, we have
\begin{equation}
\vf(Y_1^i)=(i+O(t))\frac{\partial}{\partial z}+O(t)\frac{\partial}{\partial w}+o(t)\frac{\partial}{\partial t}.\label{rep111}
\end{equation}   

Let $r(x,y,u,v,\tau,\sigma)=0$ be the equation of $M$ near the origin, where $z=x+iy$, $w=u+iv$, $t=\tau+i\sigma$ and $\hbox{grad}\, r(0)\ne 0$. Since  $\Sigma=\{t=0\}$ lies in $M$, the linear part of $r$ at the origin can only contain $\tau$ and $\sigma$, and without loss of generality we assume that $M$ is given by an equation of the form $\sigma=\varphi(x,y,u,v,\tau)$, with $\varphi(x,y,u,v,0)=0$. 

Consider the condition that the vector field $2\hbox{Re}(R)$ is tangent to $M$:
$$
2\hbox{Re}(R)(\varphi-\sigma)|_{\sigma=\varphi}=0.
$$
From this condition and formula (\ref{formRRR}) we immediately obtain
\begin{equation}
\varphi(x,y,u,v,\tau)=\tau\psi(x,y,u,v),\label{formphi}
\end{equation}
where by scaling the variable $t$ by a complex number we may suppose that $\psi(0)=0$. Next, consider the conditions that the vector fields $2\hbox{Re}(\vf(Y_j))$ for $j=0,1,2$ and $2\hbox{Re}(\vf(Y_1^i))$ are tangent to $M$:
\begin{equation}
\begin{array}{l}
2\hbox{Re}(\vf(Y_j))(\varphi-\sigma)|_{\sigma=\varphi}=0,\quad  j=0,1,2,\\
\vspace{-0.3cm}\\
2\hbox{Re}(\vf(Y_1^i))(\varphi-\sigma)|_{\sigma=\varphi}=0.
\end{array}\label{tandencyr}
\end{equation}
From the first equation with $j=1$ we see that the function $\psi$ in (\ref{formphi}) is independent of $x$, so it suffices to study identities (\ref{tandencyr}) only for $x=0$. Using formulas (\ref{X1X2}), (\ref{rep0}), (\ref{rep111}) we then deduce 
\begin{equation}
\begin{array}{l}
y\psi_y+\psi_u=0,\\
\vspace{-0.3cm}\\
\Re(e^w)\psi_u+(-2y+\Im(e^w))\psi_v=0,\\
\vspace{-0.3cm}\\ 
\psi_y=0,
\end{array}\label{syspsir}
\end{equation}
where the last equation follows by isolating the term linear in $\tau$ in the last identity in (\ref{tandencyr}). System (\ref{syspsir}) clearly yields $\psi=0$, which is impossible since in this case $M=\{\sigma=0\}$ is Levi-flat.
\vspace{0.1cm}\\  

{\bf Case 3.} Suppose now that in (\ref{rep0}) we have $f_3=0$ and the almost complex structure induced on $\Sigma$ by $M$ is $J_2^{+}$ (see (\ref{complstructures})). As in Case 2, we can assume that in (\ref{rep0}) one also has $f_1=0$, $f_2=1$.

We argue as in Case 2, but in this situation the vector field $\vf(Y_1^i)$, hence the last equation in system (\ref{syspsir}), changes. To determine the form of $\vf(Y_1^i)$, we will utilize a number of commutation relations among vector fields arising from the action of $\mathfrak{sl}(2,\C)_{\R}$ and in all such relations only focus on the coefficients at $t\partial/\partial t$. First, write $\vf(Y_2^i)$ in general form as
\begin{equation}
\vf(Y_2^i)=(iz+O(t))\frac{\partial}{\partial z}-(i-2ize^{-w}+O(t))\frac{\partial}{\partial w}+(tc(z,w)+o(t))\frac{\partial}{\partial t}\label{rep088d1}
\end{equation}
(cf.~(\ref{vfstand1})). The identity $Y_1^i=[Y_1,Y_2^i]$ then implies
\begin{equation} 
\vf(Y_1^i)=(i+O(t))\frac{\partial}{\partial z}+(2ie^{-w}+O(t))\frac{\partial}{\partial w}+\left(t\frac{\partial c}{\partial z}+o(t)\right)\frac{\partial}{\partial t}\label{rep0888d1}
\end{equation} 
(cf.~(\ref{vfstand1})). Now, from formulas (\ref{rep088d1}), (\ref{rep0888d1}) and the identities $[Y_1,Y_1^i]=0$, $[Y_1^i,Y_2]=Y_1^i$, $[Y_2,Y_2^i]=0$ we see
\begin{equation}
c=\tilde c ze^{-w}+\check c,\label{formcc}
\end{equation}                  
with $\tilde c,\check c\in\C$. Hence, we have
\begin{equation} 
\vf(Y_1^i)=(i+O(t))\frac{\partial}{\partial z}+(2ie^{-w}+O(t))\frac{\partial}{\partial w}+\left(\tilde c t e^{-w}+o(t)\right)\frac{\partial}{\partial t}.\label{rep0888d11}
\end{equation}

Expression (\ref{rep0888d11}) leads to the following analogue of system (\ref{syspsir}):
\begin{equation}
\begin{array}{l}
y\psi_y+\psi_u=0,\\
\vspace{-0.3cm}\\
\Re(e^w)\psi_u+(-2y+\Im(e^w))\psi_v=0,\\
\vspace{-0.3cm}\\ 
\psi_y-2\Im(e^{-w})\psi_u+2\Re(e^{-w})\psi_v-\Im(\tilde c \, e^{-w})(1+\psi^2)=0.
\end{array}\label{syspsirs}
\end{equation}
Resolving (\ref{syspsirs}) with respect to $\psi_y$ we get
\begin{equation}
\psi_y=\frac{(-2y+e^u\sin v)(\Im(\tilde c)\cos v-\Re(\tilde c)\sin v)(1+\psi^2)}{\sin v(4y^2-4ye^u\sin v+e^{2u})}.\label{exprpsiy}
\end{equation}
It now follows that $\Im(\tilde c)=0$ since otherwise the expression in the right-hand side of (\ref{exprpsiy}) has no limit at the origin. For $\Im(\tilde c)=0$, from (\ref{syspsirs}) we see
$$
\begin{array}{l}
\displaystyle\psi_y=-\frac{\tilde c(-2y+e^u\sin v)(1+\psi^2)}{4y^2-4ye^u\sin v+e^{2u}},\\
\vspace{-0.1cm}\\
\displaystyle\psi_u=\frac{\tilde cy(-2y+e^u \sin v)(1+\psi^2)}{4y^2-4ye^u\sin v+e^{2u}},\\
\vspace{-0.1cm}\\
\displaystyle\psi_v=-\frac{\tilde cye^u\cos v(1+\psi^2)}{4y^2-4ye^u\sin v+e^{2u}},\\
\end{array}
$$
which yields
\begin{equation}
\psi=\tan\left(\frac{\tilde c}{4}\log\left(\cos^2v+(2e^{-u}y-\sin v)^2\right)\right).\label{rel771}
\end{equation}
Clearly, we have $\tilde c\ne 0$ since otherwise $M=\{\sigma=0\}$ is Levi-flat.

We will now utilize the condition that $2\Re(\vf(V_1))$ is tangent to $M$: 
\begin{equation}
2\hbox{Re}(\vf(V_1))(\varphi-\sigma)|_{\sigma=\varphi}=0.\label{tangv}
\end{equation}
By isolating in (\ref{tangv}) the terms of order $n+1$ in $\tau$ we get
\begin{equation}
\begin{array}{l}
\displaystyle\Im\left(q(1+i\psi)^n\right)\psi_y+\Re\left(r(1+i\psi)^ne^{-w}\right)\psi_u\\
\vspace{-0.3cm}\\
\hspace{1.5cm}\displaystyle+\Im\left(r(1+i\psi)^ne^{-w}\right)\psi_v-\Im\left(s(1+i\psi)^ne^{-w}\right)(1+\psi^2)=0
\end{array}\label{rel8871}
\end{equation}
(see (\ref{rel41})). Writing $(1+i\psi)=\sqrt{1+\psi^2}e^{i\theta}$ with $\theta:=\tan^{-1}\psi$ and setting $v=0$, $\zeta:=2e^{-u}y$, from (\ref{rel771}), (\ref{rel8871}) we obtain
$$
\Im\left[\left(\left(\frac{i \tilde c r}{2}+s\right)\zeta^2+\left(\frac{\tilde c r}{2}-\tilde c q\right)\zeta+s\right)e^{in\theta}\right]=0.
$$
Hence
\begin{equation}
\tan n\theta=-\frac{\Im(\mathcal{A}) \zeta^2+\Im(\mathcal{B}) \zeta + \Im(\mathcal{C})}{\Re(\mathcal{A}) \zeta^2+\Re(\mathcal{B}) \zeta + \Re(\mathcal{C})},\label{formtan1}
\end{equation}
where
$$
\mathcal{A}:=\frac{i\tilde cr}{2}+s,\quad \mathcal{B}:=\frac{\tilde c r}{2}-\tilde c q,\quad \mathcal{C}:=s.
$$
Notice that $\mathcal{A}$, $\mathcal{B}$, $\mathcal{C}$ cannot all be equal to zero since otherwise $q=r=s=0$, which contradicts our observation (made in Case 2) that at least one of $q,r,s$ is nonzero. 

On the other hand, from (\ref{rel771}) for $v=0$ we derive
\begin{equation}
\theta=\frac{\tilde c}{4}\log(\zeta^2+1).\label{formthetasss}
\end{equation}
Hence, the range of $\theta(\zeta)$ is a half-line, and therefore the function $\tan (n\theta(\zeta))$ has infinitely many zeroes, which contradicts (\ref{formtan1}).
\vspace{0.1cm}\\

{\bf Case 4.} Suppose next that $\alpha:=f_3'(0)$ is nonzero and the almost complex structure induced on $\Sigma$ by $M$ is $J_1^{+}$ (see (\ref{complstructures}), (\ref{rep0})).

As before, we will utilize a number of commutation relations among vector fields arising from the action of $\mathfrak{sl}(2,\C)_{\R}$ and in all such relations focus on the coefficients at $t\partial/\partial t$. First of all, arguing as in Case 2, we see that (\ref{rep22}), (\ref{rep11}), (\ref{formc1}) hold and that $\tilde c_1=0$. Then, utilizing the relation $[Y_0,Y_1^i]=Y_1^i+2Y_2^i$ together with (\ref{rep0}) we obtain $\check c_1(w)=0$. Therefore, the coefficient at $\partial/\partial t$ in $\vf(Y_j^i)$ is  of the order $o(t)$ for $j=1,2$.

Next, write $\vf(Y_0^i)$ in general form
\begin{equation}
\begin{array}{l}
\displaystyle\vf(Y_0^i)=\left(-iz-iz^2+O(t)\right)\frac{\partial}{\partial z}\\
\vspace{-0.3cm}\\
\hspace{2.5cm}\displaystyle+\left(-i-2iz+ie^w+O(t)\right)\frac{\partial}{\partial w}+(t c_2(z,w)+o(t))\frac{\partial}{\partial t}
\end{array}\label{rep4455}
\end{equation}
(cf.~(\ref{rep0})). From the identities $[Y_0^i,Y_1]=Y_1^i+2Y_2^i$ and $[Y_0^i,Y_2]=-Y_0^i-Y_2^i$ together with (\ref{X1X2}), (\ref{rep22}), (\ref{rep11}), (\ref{rep4455}) we then see that $c_2=\tilde c_2 e^w$, with $\tilde c_2\in\C$, and the relation $[Y_0,Y_2^i]=-Y_0^i-Y_2^i$ yields $\tilde c_2=i\alpha$. Thus, by (\ref{rep4455}), the coefficient at $\partial/\partial t$ in $\vf(Y_0^i)$ is $i\alpha te^w+o(t)$.

Using formulas (\ref{X1X2}), (\ref{rep0}), (\ref{rep22}), (\ref{rep11}), (\ref{rep4455}), we now evaluate the vector fields $2\hbox{Re}\,(\vf(Y_j))$, $2\hbox{Re}\,(\vf(Y_j^i))$, $j=0,1,2$, at a point $p=(z,w,t)$ with $t\ne 0$ and expand the resulting vector with respect to the basis $2\hbox{Re}\,\partial/\partial z|_p$, $-2\hbox{Im}\,\partial/\partial z|_p$, $2\hbox{Re}\,\partial/\partial w|_p$, $-2\hbox{Im}\,\partial/\partial w|_p$, $2\hbox{Re}\,\partial/\partial t|_p$, $-2\hbox{Im}\,\partial/\partial t|_p$. This leads to a real $6\times 6$-matrix, say $A(p)$, and it is easy to see that $\det A(p)=|\alpha|^2\left|e^w\right|^2|t|^2+o(|t|^2)$, where $o(|t|^2)$ denotes a function of $(z,w,t)$ with the property that, being divided by $|t|^2$, for any fixed $z,w$ it tends to 0 as $t\to 0$. Hence, $A(p)$ is nondegenerate for all points $p$ sufficiently close to 0 and not lying in $\Sigma$. Thus, the local $\SL(2,\C)_{\R}$-orbit of any such $p$ is 6-dimensional, which is impossible since the local $\SL(2,\C)_{\R}$-action preserves the 5-dimensional manifold~$M$. 
\vspace{0.1cm}\\

{\bf Case 5.} Suppose finally that $\alpha:=f_3'(0)$ is nonzero and the almost complex structure induced on $\Sigma$ by $M$ is $J_2^{+}$ (see (\ref{complstructures}), (\ref{rep0})). By making the change of coordinates
$$
z\mapsto z+\frac{1-f_2(t)}{2}e^w,\quad w\mapsto w,\quad t\mapsto t,
$$
we can suppose that in (\ref{rep0}) one has $f_2=1$ (see (\ref{change88}), (\ref{change888})).

As earlier, we will utilize certain commutation relations among vector fields arising from the action of $\mathfrak{sl}(2,\C)_{\R}$ and in all of them focus on the coefficients at $t\partial/\partial t$. First of all, arguing as in Case 3, we see that (\ref{rep088d1}), (\ref{rep0888d1}), (\ref{formcc}) hold. From (\ref{rep0}) and the identity $[Y_0,Y_1^i]=Y_1^i+2Y_2^i$ we then deduce
$\tilde c=-2(\check c+i\alpha)$, which yields the formulas
\begin{equation}
\hspace{0.4cm}\makebox[250pt]{$\begin{array}{l}
\displaystyle\vf(Y_1^i)=(i+O(t))\frac{\partial}{\partial z}+(2ie^{-w}+O(t))\frac{\partial}{\partial w}+(-2(\check c+i\alpha) e^{-w}t+o(t))\frac{\partial}{\partial t},\\
\vspace{-0.3cm}\\
\displaystyle\vf(Y_2^i)=(iz+O(t))\frac{\partial}{\partial z}-(i-2ize^{-w}+O(t))\frac{\partial}{\partial w}\\
\vspace{-0.3cm}\\
\hspace{6cm}\displaystyle+\left(-2(\check c+i\alpha)ze^{-w}t+\check ct+o(t)\right)\frac{\partial}{\partial t}.
\end{array}$}\label{finalformx1ix2id}
\end{equation}

Next, from the relation $[Y_0,Y_2^i]=-Y_0^i-Y_2^i$ as well as (\ref{rep0}), (\ref{finalformx1ix2id}), it follows that
\begin{equation}
\makebox[250pt]{$\begin{array}{l}
\displaystyle\vf(Y_0^i)=-i\left(z+z^2+O(t)\right)\frac{\partial}{\partial z}\\
\vspace{-0.3cm}\\
\displaystyle\hspace{1cm} -i\left(-1-2z+e^{w}+2ze^{-w}+2z^{2}e^{-w}+O(t)\right)\frac{\partial}{\partial w}\\
\vspace{-0.3cm}\\
\displaystyle\hspace{2cm}+\left(2(\check c+i\alpha)(z+z^2)e^{-w}t+\check c(-1-2z)t-i\alpha e^wt+o(t)\right)\frac{\partial}{\partial t}
\end{array}$}\label{finalformx0id}
\end{equation}
(cf.~(\ref{vfstand1})).

We will now consider the matrix $A(z,w,t)$, as introduced in Case 4 above. From (\ref{X1X2}), (\ref{rep0}), (\ref{finalformx1ix2id}), (\ref{finalformx0id}) it is not hard to observe
$$
\det A(0,0,t)=k\,\left(\hbox{Im}\,\left(\frac{\check c}{\alpha}\right)+1\right)|t|^2+o(|t|^2),
$$ 
with $k\in\R\setminus\{0\}$. Hence, if $\hbox{Im}\,(\check c/\alpha)\ne -1$, the local $\SL(2,\C)_{\R}$-orbit of $(0,0,t)$ is 6-dimensional for all sufficiently small nonzero $|t|$, which is impossible since the local $\SL(2,\C)_{\R}$-action preserves the 5-dimensional manifold $M$. Therefore, we have
\begin{equation}
\hbox{Im}\,(\check c/\alpha)=-1.\label{relpar1}
\end{equation}

Next, as before, we see that $\vf(R)$, $\vf(V_1)$ have the forms (\ref{formRR}), (\ref{vfV}), respectively. Furthermore, the identity $[R,Y_0]=0$ together with (\ref{rep0}) yields
$$
a_1+f_3a_1'-2f_1b_1-f_1'c_1=0,\quad b_1-f_3b_1'-2a_1=0. 
$$
Since $c_1(0)=0$, it then follows that
\begin{equation}
a_1(0)=0,\quad b_1(0)=0.\label{rel1}
\end{equation}

Let $\beta:=c_1'(0)$ and consider the matrix $A_R(z,w,t)$ obtained from $A(z,w,t)$ by replacing the row corresponding to $\vf(Y_0^i)$ with that arising from $\vf(R)$. From (\ref{X1X2}), (\ref{rep0}), (\ref{formRR}), (\ref{finalformx1ix2id}), (\ref{rel1}) it is not hard to see
$$
\det A_R(0,0,t)=k\,(\hbox{Re}\,\alpha\,\hbox{Im}\,\beta-\hbox{Im}\,\alpha\,\hbox{Re}\,\beta)|t|^2+o(|t|^2),
$$
with $k\in\R\setminus\{0\}$. Hence, if $\alpha$ and $\beta$ are linearly independent over $\R$, the local $P$-orbit of $(0,0,t)$ is 6-dimensional for all sufficiently small nonzero $|t|$, which is impossible since the $P$-action preserves the 5-dimensional manifold $M$. Therefore, we have
\begin{equation}
\beta=\mu\alpha,\quad \mu\in\R. \label{relpar2}
\end{equation}

Next, the identity $[R,V_1]=V_1$ along with (\ref{formRR}), (\ref{vfV}) yields
\begin{equation}
\begin{array}{l}
\displaystyle a_2-c_1a_2'+a_1b_2+a_1'c_2=0,\\
\vspace{-0.3cm}\\
\displaystyle b_2-c_1b_2'+b_1b_2+b_1'c_2=0,\\
\vspace{-0.3cm}\\
\displaystyle c_2-c_1c_2'+b_1c_2+c_1'c_2=0
\end{array}\label{rel3}
\end{equation}
(cf.~(\ref{rel333})). As $c_2(0)=0$, relations (\ref{rel1}), (\ref{rel3}) immediately imply that if $\beta\ne 1/n$ for any $n\in\N$, then $\vf(V_1)=0$, which contradicts the effectiveness of the $\mathfrak{p}_2$-action. Therefore, $\beta=1/n$ for some $n\in\N$, in which case from (\ref{rel1}), (\ref{rel3}) we deduce that $a_2, b_2, c_2$ have the form (\ref{rel41}), where $q,r,s\in\C$. Notice that the condition $q=r=s=0$ implies $\vf(V_1)=0$, thus at least one of $q,r,s$ is nonzero (cf.~Case 2).  Also, since $\beta=1/n$, by (\ref{relpar2}) we have $\alpha\in\R$, and by (\ref{relpar1}) we write $\check c=\rho-i\alpha$ for some $\rho\in\R$.

As in Cases 2 and 3, we now express the equation of $M$ near the origin in the form $\sigma=\varphi(x,y,u,v,\tau)$, with $\varphi(x,y,u,v,0)=0$, where $z=x+iy$, $w=u+iv$, $t=\tau+i\sigma$. We then write $\varphi(x,y,u,v,\tau)=\tau\psi(x,y,u,v)+o(\tau)$\footnote{Here $o(\tau)$ denotes any function of the form $\sum_{j=2}^{\infty}a_j(x,y,u,v)\tau^j$, where $a_j(x,y,u,v)$ is real-analytic.} and focus on the function $\psi$, for which we may assume $\psi(0)=0$. Next, as in (\ref{tandencyr}), consider the conditions that the vector fields $2\hbox{Re}(\vf(Y_j))$ for $j=0,1,2$ and $2\hbox{Re}(\vf(Y_1^i))$ are tangent to $M$ and single out the terms linear in $\tau$. From the first equation in (\ref{tandencyr}) with $j=1$ we immediately see that $\psi$ is independent of $x$, so it suffices to consider identities (\ref{tandencyr}) only for $x=0$. Using formulas (\ref{X1X2}), (\ref{rep0}), (\ref{finalformx1ix2id}) we then obtain 
\begin{equation}
\begin{array}{l}
y\psi_y+\psi_u=0,\\
\vspace{-0.3cm}\\
\Re(e^w)\psi_u+(-2y+\Im(e^w))\psi_v-\alpha\Im(e^w)(1+\psi^2)=0,\\
\vspace{-0.3cm}\\ 
\psi_y-2\Im(e^{-w})\psi_u+2\Re(e^{-w})\psi_v+2\rho\Im(e^{-w})(1+\psi^2)=0
\end{array}\label{syspsi}
\end{equation}
(cf.~(\ref{syspsir}), (\ref{syspsirs})). System of equations (\ref{syspsi}) can be solved explicitly, and we get
\begin{equation}
\psi=\tan\left(-\alpha\tan^{-1}\frac{2e^{-u}y-\sin v}{\cos v}-\frac{\rho}{2}\log\left(\cos^2v+(2e^{-u}y-\sin v)^2\right)\right)\label{rel77}
\end{equation}
(cf.~(\ref{rel771})).

We will now utilize condition (\ref{tangv}) of the tangency of $2\Re(\vf(V_1))$ to $M$. As in Case 3, this leads to (\ref{rel8871}), hence, for for $\theta:=\tan^{-1}\psi$, $v=0$, $\zeta:=2e^{-u}y$, to 
$$
\Im\left[\Bigl((i\rho r-s)\zeta^2+(i\alpha r+\rho r-2\rho q)\zeta+(\alpha r-2\alpha q-s)\Bigr)e^{in\theta}\right]=0\label{formmmm8888}
$$
and to (\ref{formtan1}), where
\begin{equation}
\mathcal{A}:=i\rho r-s,\quad \mathcal{B}:=i\alpha r+\rho r-2\rho q,\quad \mathcal{C}:=\alpha r-2\alpha q-s.\label{coeffmathcal1}
\end{equation}
Notice that $\mathcal{A}$, $\mathcal{B}$, $\mathcal{C}$ cannot all be equal to zero since otherwise $q=r=s=0$, which contradicts our observation that at least one of $q,r,s$ is nonzero. 

On the other hand, from (\ref{rel77}) for $v=0$ we obtain
\begin{equation}
\theta=-\alpha\tan^{-1}\zeta-\frac{\rho}{2}\log(\zeta^2+1)\label{exprtheta}
\end{equation}
(cf.~(\ref{formthetasss})). Hence if $\rho\ne 0$, the range of $\theta(\zeta)$ is a half-line, and therefore the function $\tan (n\theta(\zeta))$ has infinitely many zeroes, which contradicts (\ref{formtan1}). Thus, it follows that $\rho=0$. Furthermore, by an elementary analysis of the graph of the function $\tan(n\theta(\zeta))$ one notes that the number of its zeroes (resp.~poles) is odd (resp.~even) and using (\ref{formtan1}), (\ref{exprtheta}) deduces that either $\alpha=\pm 1/n$ or $\alpha=\pm 2/n$. We will now consider each of these four possibilities separately. Before proceeding, we record the formulas for the first-order partial derivatives of $\psi$:
$$
\begin{array}{l}
\displaystyle\psi_y=-\frac{2\alpha e^u\cos v(1+\psi^2)}{4y^2-4ye^u\sin v+e^{2u}},\\
\vspace{-0.1cm}\\
\displaystyle\psi_u=\frac{2\alpha y e^u\cos v(1+\psi^2)}{4y^2-4ye^u\sin v+e^{2u}},\\
\vspace{-0.1cm}\\
\displaystyle\psi_v=-\frac{\alpha(2ye^u\sin v-e^{2u})(1+\psi^2)}{4y^2-4ye^u\sin v+e^{2u}},
\end{array}
$$
which follow directly from (\ref{syspsi}). 

{\bf Case 5.1.} Suppose first that $\alpha=1/n$. In this case formulas (\ref{formtan1}), (\ref{coeffmathcal1}) yield $q=0$, $s=i\Im(r)/n$, where in (\ref{formtan1}) we took into account that for $v=0$ one has $\tan (n\theta(\zeta))=-\zeta$. Further, (\ref{rel8871}) simplifies to
\begin{equation}
(2e^{-u}y-\sin v)(P\sin v+Q\cos v)-(P\cos v-Q\sin v)\cos v=0,\label{ids777}
\end{equation}
where
$$
\begin{array}{l}
\displaystyle P:=2\left(2\Im(r)e^{-u}y^2-\Re(r)y\cos v-\Im(r)y\sin v\right),\\
\vspace{-0.1cm}\\
\displaystyle Q:=2\Im(r)y\cos v+\Re(r)(2y\sin v-e^u)
\end{array}
$$
and where we utilized
$$
\tan n\theta=-\frac{2e^{-u}y-\sin v}{\cos v}.
$$
It is immediate from (\ref{ids777}) that $\Im(r)=0$. Hence we have $s=0$, and the last equation in (\ref{rel3}) together with (\ref{rel41}) implies $c_2=0$. It then follows from (\ref{vfV}), (\ref{rel41}) that
\begin{equation}
\vf(V_1)=o(t^n)\frac{\partial}{\partial z}+(r t^n+o(t^n))e^{-w}\frac{\partial}{\partial w},\label{vf11}
\end{equation}
where $r\in\R$.

Next, by (\ref{relations8888}) we have $\vf(V_2)=[\vf(Y_0)+\vf(Y_2),\vf(V_1)]$, which with the help of (\ref{X1X2}), (\ref{rep0}), (\ref{vf11}) yields
\begin{equation}
\vf(V_2)=o(t^n)\frac{\partial}{\partial z}+(rt^n(2ze^{-w}-1)+o(t^n))\frac{\partial}{\partial w}-\left(r\frac{t^{n+1}}{n}+o(t^{n+1})\right)\frac{\partial}{\partial t}.\label{vf21}
\end{equation}
Further, by (\ref{relations8888}) we have $\vf(V_3)=-[\vf(Y_2^i),\vf(V_2)]$, which together with (\ref{finalformx1ix2id}), (\ref{vf21}) implies
$$
\vf(V_3)=o(t^n)\frac{\partial}{\partial z}-(irt^n+o(t^n))\frac{\partial}{\partial w}-\left(ir\frac{t^{n+1}}{n}+o(t^{n+1})\right)\frac{\partial}{\partial t}.
$$
Finally, by (\ref{relations8888}) we have $\vf(V_4)=[\vf(Y_0)+\vf(Y_2),\vf(V_2)]/2$, and using (\ref{X1X2}), (\ref{rep0}), (\ref{vf21}) obtain
$$
\vf(V_4)=o(t^n)\frac{\partial}{\partial z}+(rzt^n(ze^{-w}-1)+o(t^n))\frac{\partial}{\partial w}-\left(rz\frac{t^{n+1}}{n}+o(t^{n+1})\right)\frac{\partial}{\partial t}.
$$

Observe now that the examples from Theorem \ref{main2} are obtained from the situation of Case 5.1 by taking into account only the linear term with respect to $\tau$ in the expression for the defining function $\varphi$ of $M$ and only the terms of the lowest possible orders with respect to $t$ in the expressions for the vector fields computed here, as well as by applying the following automorphism of $\mathfrak{p}_2$:
$$
\begin{array}{llll}
X_1\mapsto -X_3, & X_2\mapsto -X_2, & X_3\mapsto -X_1, &\\
\vspace{-0.1cm}\\
X_1^i\mapsto -X_3^i, & X_2^i\mapsto -X_2^i,& X_3^i\mapsto -X_1^i, & R\mapsto R,\\
\vspace{-0.1cm}\\
V_1\mapsto V_4, & V_2\mapsto V_2, & V_3\mapsto -V_3, & V_4\mapsto V_1.
\end{array}
$$
\vspace{-0.1cm}\\

{\bf Case 5.2.} Suppose next that $\alpha=-1/n$. In this case formulas (\ref{formtan1}), (\ref{coeffmathcal1}) yield $q=r$, $s=i\Im(r)/n$, where in (\ref{formtan1}) we took into account that for $v=0$ one has $\tan (n\theta(\zeta))=\zeta$. Further, (\ref{rel8871}) simplifies to
\begin{equation}
\begin{array}{l}
(2e^{-u}y-\sin v)(2\Re(r)e^u\cos v+P\sin v+Q\cos v))+\\
\vspace{-0.3cm}\\
\hspace{2.5cm}(2\Im(r)e^u\cos v+P\cos v-Q\sin v)\cos v=0,
\end{array}\label{ids7771}
\end{equation}
where
$$
\begin{array}{l}
P:=-2(2\Im(r)e^{-u}y^2-3\Im(r)y\sin v+\Re(r)y\cos v+\Im(r)e^{u}),\\
\vspace{-0.3cm}\\
Q:=2\Im(r)y\cos v+\Re(r)(2y\sin v-e^u)
\end{array}
$$
and where we utilized
$$
\tan n\theta=\frac{2e^{-u}y-\sin v}{\cos v}.
$$
It is immediate from (\ref{ids7771}) that $\Im(r)=0$. Hence we have $s=0$, and the last equation in (\ref{rel3}) together with (\ref{rel41}) again implies $c_2=0$. It then follows from (\ref{vfV}), (\ref{rel41}) that
$$
\vf(V_1)=(rt^n+o(t^n))\frac{\partial}{\partial z}+(rt^n+o(t^n))e^{-w}\frac{\partial}{\partial w},
$$
where $r\in\R$. Further, arguing as in Case 5.1 we obtain expressions for $\vf(V_2)$, $\vf(V_3)$, $\vf(V_4)$:
$$
\begin{array}{l}
\displaystyle \vf(V_2)=(rt^n(2z-e^w)+o(t^n))\frac{\partial}{\partial z}+(rt^n(2ze^{-w}-1)+o(t^n))\frac{\partial}{\partial w}\\
\vspace{-0.3cm}\\
\hspace{8.2cm}\displaystyle+\left(r\frac{t^{n+1}}{n}+o(t^{n+1})\right)\frac{\partial}{\partial t},\\
\vspace{-0.1cm}\\
\displaystyle \vf(V_3)=-(irt^ne^w+o(t^n))\frac{\partial}{\partial z}-(irt^n+o(t^n))\frac{\partial}{\partial w}-\left(ir\frac{t^{n+1}}{n}+o(t^{n+1})\right)\frac{\partial}{\partial t},\\
\vspace{-0.1cm}\\
\displaystyle \vf(V_4)=(rt^nz(z-e^w)+o(t^n))\frac{\partial}{\partial z}+(rt^nz(ze^{-w}-1)+o(t^n))\frac{\partial}{\partial w}\\
\vspace{-0.3cm}\\
\hspace{7cm}\displaystyle+\left(r(z-e^w)\frac{t^{n+1}}{n}+o(t^{n+1})\right)\frac{\partial}{\partial t}.
\end{array}
$$

As in Case 5.1, we can now consider the linear term with respect to $\tau$ in the expression for the defining function $\varphi$ of $M$ as well as the terms of the lowest possible orders with respect to $t$ in the expressions for the corresponding vector fields, and attempt to construct further examples in the spirit of Theorem \ref{main2}. In this case the manifold, say $M_n'\subset\C^3$ is defined for $-\pi/2<v<\pi/2$ by the equation
\begin{equation}
\sigma=\tau\tan\left(\frac{1}{n}\tan^{-1}\frac{2y-e^u\sin v}{e^u\cos v}\right).\label{eqexamp1}
\end{equation}
It is then clear from (\ref{eqexamp}) and (\ref{eqexamp1}) that $M_n$ and $M_n'$ for every $n$ are CR-equivalent by means of the mapping
$$
z\mapsto -z+e^w-1,\quad w \mapsto w,\quad t \mapsto t,
$$
which is biholomorphic on $\{-\pi/2<v<\pi/2\}$ and preserves the origin. Thus, passing to the lowest-order terms in Case 5.2 leads to the examples given earlier.
\vspace{-0.1cm}\\

{\bf Case 5.3.} Suppose next that $\alpha=2/n$. In this case formulas (\ref{formtan1}), (\ref{coeffmathcal1}) yield $q=0$, $r\in\R$, $s=r/n$, where in (\ref{formtan1}) we took into account that for $v=0$ one has
$$
\tan (n\theta(\zeta))=\frac{2\zeta}{\zeta^2-1}.
$$
Further, (\ref{rel8871}) simplifies to
\begin{equation}
\hspace{0.4cm}\makebox[250pt]{$\begin{array}{l}
r\left((4e^{-2u}y^2-4e^{-u}y\sin v+\sin^2 v-\cos^2 v)(4y^2\sin v+4ye^u\cos^2 v-e^{2u}\sin v)\right.\\
\vspace{-0.3cm}\\
\hspace{3.5cm}\left.-2\cos^2 v(2e^{-u}y-\sin v)(4y^2-4ye^u\sin v-e^{2u})\right)=0,
\end{array}$}\label{ids7772}
\end{equation}
where we utilized
$$
\tan n\theta=\frac{2\cos v(2e^{-u}y-\sin v)}{4e^{-2u}y^2-4e^{-u}y\sin v+\sin^2 v-\cos^2 v}.
$$
Identity (\ref{ids7772}) immediately implies $r=0$, hence $q=r=s=0$, which contradicts our earlier observation that at least one of $q,r,s$ must be nonzero.
\vspace{-0.1cm}\\

{\bf Case 5.4.} Suppose finally that $\alpha=-2/n$. In this case formulas (\ref{formtan1}), (\ref{coeffmathcal1}) yield $q=r$, $r\in\R$, $s=r/n$, where in (\ref{formtan1}) we took into account that for $v=0$ one has
$$
\tan (n\theta(\zeta))=-\frac{2\zeta}{\zeta^2-1}.
$$
Further, (\ref{rel8871}) simplifies to
\begin{equation}
\hspace{0.3cm}\makebox[250pt]{$\begin{array}{l}
\hspace{0.3cm} r\left((4e^{-2u}y^2-4e^{-u}y\sin v+\sin^2 v-\cos^2 v)(4y^2\sin v-4ye^u\cos^2 v\right.\\
\vspace{-0.3cm}\\
\hspace{0cm}
\left. -8ye^u\sin^2 v +3e^{2u}\sin v)+2\cos^2 v(2e^{-u}y-\sin v)(4y^2-4ye^u\sin v-e^{2u})\right)=0,
\end{array}$}\label{ids7773}
\end{equation}
where we utilized
$$
\tan n\theta=-\frac{2\cos v(2e^{-u}y-\sin v)}{4e^{-2u}y^2-4e^{-u}y\sin v+\sin^2 v-\cos^2 v}.
$$
Identity (\ref{ids7773}) immediately yields $r=0$, hence $q=r=s=0$, which again contradicts our earlier observation that at least one of $q,r,s$ must be nonzero.
\vspace{-0.1cm}\\

We now summarize our findings as follows:

\begin{theorem}\label{examplegeneral} If the isotropy subalgebras of the points in the local $\SL(2,\C)_{\R}$-orbit $\Sigma$ are conjugate to $\mathfrak{c}_{\R}$, only Cases {\rm 5.1}, {\rm 5.2} contain examples of CR-hypersur\-fa\-ces with $\dim\mathfrak{s}=11$. In fact, the examples in Theorem {\rm \ref{main2}} arise from each of the two cases by collecting the linear term in the expression for the defining function of $M$ with respect to $\tau$ and  and the terms of the lowest possible orders with respect to $t$ in the expressions for the corresponding vector fields.
\end{theorem}

\begin{remark}\label{higherdegreeexamples} It would be interesting to see whether one can find more examples of manifolds with $\dim\mathfrak{s}=11$ (or even fully classify all such manifolds) by allowing higher-order terms in the expressions for $\varphi$ and the vector fields obtained in Cases 5.1 and 5.2.
\end{remark}

\subsection{Other orbit types}\label{caseother}

In accordance with Proposition \ref{isotropysubalgebras} and Lemma \ref{cartancomplexorbit}, we have also considered two other cases: (i) when the isotropy subalgebras of the points in the local $\SL(2,\C)_{\R}$-orbit $\Sigma$ are conjugate to $\mathfrak{n}_{\R}$ and $\Sigma$ is not a complex surface, and (ii) when the isotropy subalgebras of the points in $\Sigma$ are conjugate to $\mathfrak{b}_{\R}$ and $\Sigma$ is a totally real surface. Both these situations lead to contradictions similar to the ones we encountered in Cases 1--4 above, thus no examples of manifolds with $\dim\mathfrak{s}=11$ arise from them. We do not provide details of our calculations here. To complete the picture, one also needs to look at two more possibilities: (iii) when the isotropy subalgebras of the points in $\Sigma$ are conjugate to $\mathfrak{n}_{\R}$ and $\Sigma$ is a complex surface, and (iv) when the isotropy subalgebras of the points in $\Sigma$ are conjugate to $\mathfrak{b}_{\R}$  and $\Sigma$ is a complex curve. A quick look at (iii) shows that this case can be analyzed analogously to the case of Cartan subalgebras in Section \ref{caseCartan} although we did not perform all the required calculations in detail. In contrast, (iv) appears to be quite hard as in this situation only one vector field can be brought to a canonical form, say $\partial/\partial z$, and the two remaining variables leave significant freedom for reducing the generality of other vector fields (cf.~\cite[Chapter I, \S 4--5]{IY}). We did not attempt to analyze this case.

In particular, we have:  

\begin{proposition}\label{complexstructure} Let $M$ be a CR-hypersurface with $\mathfrak{s}=\mathfrak{p}_2$ as in Theorem {\rm \ref{main1}}. Further, let  $O$ be a local $P$-orbit of positive codimension in $M$, where the group $P=(\SL(2,\C)_{\R}\times\RR)\ltimes\R^4$ has $\mathfrak{p}_2$ as its Lie algebra {\rm (}such an orbit always exists if $M$ is simply-connected{\rm )}. Then $O$ contains a complex submanifold. More precisely, fix a point $p_0\in O$ and consider the local orbit $\Sigma$ of $p_0$ under the induced local action of $\SL(2,\C)_{\R}\subset P$. Then $\Sigma$ is either a complex surface or a complex curve. In the former case the isotropy subalgebra of $p_0$ in $\mathfrak{sl}(2,\C)_{\R}$ is conjugate to one of $\mathfrak{c}_{\R}$, $\mathfrak{n}_{\R}$, and in the latter case to $\mathfrak{b}_{\R}$.
\end{proposition}

\section{Examples of CR-hypersurfaces with $\dim\mathfrak{s}<11$}\label{moreexamples}
\setcounter{equation}{0}

In this section we show that all integers between 0 and 10 are also realizable as symmetry dimensions thus establishing that no more gaps occur for the dimension of the symmetry algebra of a real-analytic connected holomorphically nondegenerate {\rm 5}-dimensional CR-hypersurface.

\begin{theorem}\label{realizableimensions} For a real-analytic connected holomorphically nondegenerate {\rm 5}-dimensional CR-hypersurface the possible symmetry dimensions are {\rm 0--11} and {\rm 15}, all of them are realizable, and for each of these dimensions there is a simply-connected realization.
\end{theorem}

\begin{proof} We only need to realize symmetry dimensions 0--10. There are several ways of doing that and many examples are well-known (see, e.g., \cite{IZ}, \cite{Kr} and references therein for a discussion of dimensions 7, 8, 10). In the proof below, we choose a uniform approach to realizing dimensions 2--8, 10 by considering holomorphically nondegenerate hypersurfaces of the form
\begin{equation}
\sigma=P(z,\bar{z},w,\bar{w}),\label{polynom}
\end{equation}
where, as before, $z,w,t$ are complex coordinates in $\C^3$ with $t=\tau+i\sigma$, and $P$ a homogeneous polynomial without pluriharmonic terms of degree $d\ge 2$. Explicit formulas for the components of the symmetry algebra of hypersurface (\ref{polynom}) were given in \cite[Theorem 1.1]{KMZ}, and utilizing {\tt Maple} one can compute their dimensions for particular examples. 

We experimented with degrees $d=3, 4$ and obtained the following examples realizing symmetry dimensions 2--8, 10:
$$
\begin{array}{ll}
\dim\mathfrak{s}=2, & P=|z|^2(z+\bar{z})+|w|^2(-iz+w+i\bar{z}+\bar{w})+z\bar{w}^2+w^2\bar{z},\\
\vspace{-0.1cm}\\
\dim\mathfrak{s}=3, & P=|z|^2(w+\bar{w})+iz\bar{w}^2-iw^2\bar{z},\\
\vspace{-0.1cm}\\
\dim\mathfrak{s}=4, & P=|z|^2(z+\bar{z})+|w|^2(w+\bar{w}),\\
\vspace{-0.1cm}\\
\dim\mathfrak{s}=5, & P=z^2\bar{w}^2+w^2\bar{z}^2,\\
\vspace{-0.1cm}\\
\dim\mathfrak{s}=6, & P=|z|^2(w^2+\bar{w}^2),\\
\vspace{-0.1cm}\\
\dim\mathfrak{s}=7, & P=(|z|^2+|w|^2)^2,\\
\vspace{-0.1cm}\\
\dim\mathfrak{s}=8, & P=|z|^2(w+\bar{w}),\\
\vspace{-0.1cm}\\
\dim\mathfrak{s}=10, & P=z\bar{w}^2+w^2\bar{z}.
\end{array}
$$
Notice that for each of the polynomials listed above the corresponding hypersurface of the form (\ref{polynom}) is holomorphically nondegenerate since it has points of Levi-nondegeneracy.

Next, for the sphere $S^5\subset\CC^3$ consider the lens space ${\mathcal L}_m:=S^5/\ZZ_m$, with $m>1$, where $\Z_m$ acts on $\CC^3$ by complex multiplication. It is not hard to show (see, e.g., \cite[p.~37]{I1}) that the Lie group of CR-automorphisms of ${\mathcal L}_m$ is naturally  isomorphic to $\U(3)/\ZZ_m$ and thus has dimension 9. Clearly, its Lie algebra consists of complete vector fields in $\mathfrak{s}=\hol({\mathcal L}_m)$. On the other hand, as ${\mathcal L}_m$ is compact, every vector field in $\mathfrak{s}$ is complete, thus every ${\mathcal L}_m$ is an example of a Levi-nondegenerate manifold with symmetry dimension 9.

There also exist simply-connected examples with symmetry dimension 9. For instance, utilizing {\tt Maple}, one can compute the symmetry algebra of the hypersurface
\begin{equation}
\sigma=|z|^2+|w|^4.\label{polynom1}
\end{equation}
The algebra is spanned by the following 9 vector fields:
$$
\begin{array}{l}
\displaystyle iz\frac{\partial}{\partial z},\,\,iw\frac{\partial}{\partial w},\,\,\frac{\partial}{\partial t},\\
\vspace{-0.1cm}\\
\displaystyle \frac{z}{2}\frac{\partial}{\partial z}+\frac{w}{4}\frac{\partial}{\partial w}+t\frac{\partial}{\partial t},\\
\vspace{-0.1cm}\\
\displaystyle {zt}\frac{\partial}{\partial z}+\frac{wt}{2}\frac{\partial}{\partial w}+t^2\frac{\partial}{\partial t},\\
\vspace{-0.1cm}\\
\displaystyle \frac{i}{2}\frac{\partial}{\partial z}+z\frac{\partial}{\partial t},\,\,\frac{1}{2}\frac{\partial}{\partial z}+iz\frac{\partial}{\partial t},\\
\vspace{-0.1cm}\\
\displaystyle\left(z^2+\frac{it}{2}\right)\frac{\partial}{\partial z}+\frac{zw}{2}\frac{\partial}{\partial w}+zt\frac{\partial}{\partial t},\\
\vspace{-0.1cm}\\
\displaystyle\left(iz^2+\frac{t}{2}\right)\frac{\partial}{\partial z}+\frac{izw}{2}\frac{\partial}{\partial w}+izt\frac{\partial}{\partial t}.
\end{array}
$$
We remark that the symmetry algebras of hypersurfaces of the form (\ref{polynom}) where $P$ is a weighted homogeneous polynomial (e.g., as in (\ref{polynom1})) were studied in \cite{KMZ}, \cite{KM}.

Next, notice that the symmetry dimension of a hypersurface of the form (\ref{polynom}) is always at least $2$ due to the presence of the translations
$$
z\mapsto z,\quad w\mapsto w,\quad t\mapsto t+a,\,a\in\RR
$$
and the dilations
$$
z\mapsto b z,\quad w\mapsto b w,\quad t\mapsto b^dt,\, b>0.
$$
Therefore, to realize 0 and 1 as symmetry dimensions one has to utilize equations of a different kind. Indeed, a generic hypersurface of the form $\sigma=F(|z|^2,w,\bar{w},\tau)$ has 1-dimensional symmetry algebra, and a generic hypersurface $\sigma=F(z,\bar{z},w,\bar{w},\tau)$ has no symmetries at all. \end{proof}

\end{document}